\documentclass[11pt]{article}
\usepackage{amscd,amsfonts,amsmath,amssymb,amsthm}
\usepackage{picins}
\usepackage{slashbox}
\usepackage{booktabs,cases,cite,color,diagbox}
\usepackage{enumitem,fancyhdr,graphicx}
\usepackage{accents,hyperref,indentfirst}
\usepackage{leftidx,mathrsfs,multirow}
\usepackage{picins,ragged2e,rotating,slashbox}
\usepackage{verbatim}
\usepackage[margin=1.5in]{geometry}
\usepackage[all]{xy}
\newtheorem{definition}{Definition}[section]
\newtheorem{theorem}{Theorem}[section]
\newtheorem{lemma}{Lemma}[section]

\newtheorem{remark}{Remark}[section]

\numberwithin{equation}{section}

\begin{document}

\title{\bf The capillary Orlicz-Minkowski problem}
\date{}

\author{Xudong Wang, Baocheng Zhu}

\maketitle

\begin{abstract}

In this paper, we introduce a Robin boundary analogue of the Orlicz-Minkowski problem, which seeks to find a capillary convex body with a prescribed capillary Orlicz surface area measure in the upper Euclidean half-space. We obtain the volume-normalized smooth solutions to the capillary even Orlicz-Minkowski problem by the continuity method. In addition, we also establish a capillary Orlicz-Brunn-Minkowski inequality and a capillary Orlicz-Minkowski inequality, which can change our solutions to a spherical cap under some conditions.

\end{abstract}

\textbf{MSC2020:} 53C42, 53C45, 35J66.

\textbf{Keywords:} Capillary hypersurface, Capillary Orlicz-Minkowski problem, Robin boundary value condition, Capillary Orlicz-Brunn-Minkowski inequality.

\section{Introduction}

Let $\mathbb{R}^{n+1}_+=\{x\in\mathbb{R}^{n+1}\,|\,x_{n+1}>0\}$ be the upper Euclidean half-space. We call a hypersurface $\Sigma$ in $\mathbb{R}^{n+1}_+$ with boundary $\partial\Sigma\subset\partial\mathbb{R}^{n+1}_+$ capillary if it intersects with $\partial\mathbb{R}^{n+1}_+$ at a constant contact angle $\theta\in(0,\pi)$. If $\Sigma$ is $C^2$-smooth convex hypersurface in $\mathbb{R}^{n+1}_+$ with positive curvature, then the domain $\widehat{\Sigma}$ bounded by $\Sigma$ and $\partial\mathbb{R}^{n+1}_+$ is called a capillary convex body, and we denote by $|\widehat{\Sigma}|=V(\widehat{\Sigma})$ the volume of $\widehat{\Sigma}$. Denote by $\mathcal{K}_\theta$ the family of all capillary convex bodies in $\overline{\mathbb{R}^{n+1}_+}$, and by $\mathcal{K}_\theta^\circ$ the subfamily of $\mathcal{K}_\theta$ whose elements contain the origin in the interior of their flat boundary.

A simple example of a capillary convex hypersurface is the following spherical cap $\mathcal{C}_\theta$:
\begin{equation*}
\mathcal{C}_\theta=\{\xi\in\overline{\mathbb{R}^{n+1}_+}\,:\,|\xi-\cos\theta e|=1\},
\end{equation*}
where $e=(0,\cdots,0,-1)$. Then, the capillary Gauss map is defined by
\begin{equation*}
\begin{aligned}
\widetilde{\nu}:\Sigma&\to\mathcal{C}_\theta\\
X&\mapsto\nu(X)+\cos\theta e,
\end{aligned}
\end{equation*}
where $\nu$ is the usual Gauss map. It is easy to check that the capillary Gauss map $\widetilde{\nu}$ is a diffeomorphism map between $\Sigma$ and $\mathcal{C}_\theta$.

In \cite{Mei-Wang-Weng-The_capillary_Minkowski_problem}, Mei, Wang and Weng proposed a capillary Minkowski problem, which asks for a capillary convex body $\widehat{\Sigma}\in\mathcal{K}_\theta$ such that its the Gauss-Kronecker curvature satisfies
\begin{equation*}
K(\widetilde{\nu}^{-1}(\xi))=f(\xi),\ \ \forall\xi\in\mathcal{C}_\theta,
\end{equation*}
for a given a positive, smooth function $f$ on $\mathcal{C}_\theta$. By the continuity method as Cheng-Yau \cite{Cheng-Yau-Regularity_Minkowski_Problem} and Lions-Trudinger-Urbas \cite{Lions-Trudinger-Urbas-The_Neumann_problem_for_equations_of_Monge_Ampere_type}, they obtained smooth solutions to the capillary Minkowski problem. It is very meaningful that the capillary Minkowski problem is actually a natural Robin boundary version of the classical Minkowski problem. For the solutions to the classical Minkowski problem, one can refer to Minkowski \cite{Minkowski-konvexen_Polyeder}, Aleksandrov \cite{Aleksandrov-surface_area_measure}, Fenchel and Jessen \cite{Fenchel-Jessen}, Lewy \cite{Lewy-differentialgeometricinlarge}, Nirenberg \cite{Nirenberg}, Pogorelov \cite{Pogorelovbook}, Cheng and Yau \cite{Cheng-Yau-Regularity_Minkowski_Problem}, Caffarelli \cite{Caffarelli-Interior_W_Estimates,Caffarelli-viscosity_solutions}.

In the 1990s, Lutwak \cite{Lutwak-The_Brunn_Minkowski_Firey_Theory_I} introduced the $L_p$-Minkowski problem using Firey's $p$-sum \cite{Firey-p_Means_of_convex_bodies}. The $L_p$-Minkowski problem has become one of the core problems in convex geometry and geometric partial differential equations, which also includes the famous logarithmic Minkowski problem and centroaffine Minkowski problem. On the solutions to the $L_p$-Minkowski problem, one can refer to \cite{Lutwak-The_Brunn_Minkowski_Firey_Theory_I,Lutwak-Oliker,Bianchi-Boroczky-Colesanti-Smoothness_in_the_L_p_Minkowski_problem_for_p_1,
Bianchi-Boroczky-Colesanti-Yang-The_L_p_Minkowski_problem_for_-n_p_1,Boroczky-Lutwak-Yang-Zhang-The_logarithmic_Minkowski_problem,
Bryan-Ivaki-Scheuer-A_unified_flow_approach_to_smooth_even_L_p_Minkowski_problems,Chen-Li-Zhu-On_the_L_p_Monge_Ampere_equation,
Chou-Wang-The_L_p_Minkowski_problem_and_the_Minkowski_problem_in_centroaffine_geometry,
Guang-Li-Wang-Existence_of_convex_hypersurfaces_with_prescribed_centroaffine_curvature,
Guang-Li-Wang-The_L_p_Minkowski_problem_with_super_critical_exponents,Guan-Lin-On_the_equation_det_on_S_n,
He-Li-Wang-Multiple_solutions_of_the_Lp_Minkowski_problem,Huang-Lu-On_the_regularity_of_the_Lp_Minkowski_problem,
Hug-Lutwak-Yang-Zhang-On_the_L_p_Minkowski_problem_for_polytopes,Jian-Lu-Wang-Nonuniqueness_of_solutions_to_the_Lp_Minkowski_problem,
Lutwak-Yang-Zhang-On_the_L_p_Minkowski_problem,Lu-Wang-Rotationally_symmetric_solutions_to_the_Lp_Minkowski_problem,
Zhu-The_centro_affine_Minkowski_problem_for_polytopes,Zhu-The_logarithmic_Minkowski_problem_for_polytopes,
Zhu-The_Lp_Minkowski_problem_for_polytopes_for_0_p_1,Zhu-The_Lp_Minkowski_problem_for_polytopes_for_0_p_1}, and the references in. Similar to the methods in Guan-Lin \cite{Guan-Lin-On_the_equation_det_on_S_n}, Mei, Wang and Weng obtained smooth solutions to the capillary $L_p$-Minkowski problem for $p\geqslant1$. Furthermore, the case of $-n<p<1$ was solved by Hu and Ivaki \cite{Hu-Ivaki-Capillary_curvature_images}. For more details on the capillary convex hypersurfaces, one can refer to \cite{Mei-Wang-Weng-Convex_capillary_hypersurfaces_of_prescribed_curvature_problem,
Hu-Ivaki-Scheuer-Capillary_Christoffel_Minkowski_problem,Mei-Wang-Weng-The_capillary_Gauss_curvature_flow}, and the references in.

As a further extension of the $L_p$-Minkowski problem, Haberl, Lutwak, Yang and Zhang \cite{Haberl-Lutwak-Yang-Zhang-Orlicz_Minkowski_problem} raised the Orlicz-Minkowski problem. In the smooth case, it is equivalent to solve the Monge-Amp\`{e}re type equation on the unit sphere $\mathbb{S}^n$:
\begin{equation*}
\varphi(h)\det(h_{ij}+h\delta_{ij})=f,
\end{equation*}
where $h$ is the support function of a convex set, $h_{ij}$ represents the twice covariant derivative of $h$ with respect to the standard round metric on $\mathbb{S}^n$, $\varphi$ and $f$ are some given smooth functions on $\mathbb{R}$ and $\mathbb{S}^n$ respectively. When $\varphi(x)=x^{1-p}$, this Orlicz-Minkowski problem is the $L_p$-Minkowski problem. Related to this, the classical Brunn-Minkowski theory has been extended to the Orlicz setting, that is, the Orlicz-Brunn-Minkowski theory and dual Orlicz-Brunn-Minkowski theory. For more details, one can refer to \cite{Gardner-Hug-Weil-The_Orlicz_Brunn_Minkowski_theory,Gardner-Hug-Weil-Ye-The_dual_Orlicz_Brunn_Minkowski_theory,
Lutwak-Yang-Zhang-Orlicz_centroid_bodies,Lutwak-Yang-Zhang-Orlicz_projection_bodies,Zhu-Zhou-Xu-Dual_Orlicz_Brunn_Minkowski_theory,
Jian-Lu-Existence_of_solutions_to_the_Orlicz_Minkowski_problem,Liu-Lu-A_flow_method_for_the_dual_Orlicz_Minkowski_problem,
Zou-Xiong-Orlicz_John_ellipsoids}, and the references in. And for other Minkowski type problems, please see e.g., \cite{Boroczky-Lutwak-Yang-Zhang-Zhao-The_Gauss_Image_Problem,Schneider-book,Huang-Yang-Zhang-Minkowski_problems_for_geometric_measures,
Huang-Lutwak-Yang-Zhang-L_p_Aleksandrov_Problem,Lutwak-Yang-Zhang-Lp_Dual_curvature_measures,Huang-Lutwak-Yang-Zhang-dual_curvature_measures}.

In this paper, we consider the corresponding capillary version of the Orlicz-Minkowski problem. Meanwhile, we also provide a framework of the Orlicz-Brunn-Minkowski theroy for capillary convex bodies. For two capillary convex bodies $\widehat{\Sigma}_1,\widehat{\Sigma}_2\in\mathcal{K}_\theta^\circ$, we will define their Orlicz sum and derive the corresponding Orlicz-Brunn-Minkowski inequality in Section \ref{Preliminaries}. The support function of a capillary convex body $\widehat{\Sigma}\in\mathcal{K}_\theta$ is defined by
\begin{equation*}
h(\xi)=\langle\widetilde{\nu}^{-1}(\xi),\xi-\cos\theta e\rangle, \ \xi\in\mathcal{C}_\theta.
\end{equation*}
In particular, the support function of $\mathcal{C}_\theta$ is $\ell(\xi)=\sin^2\theta+\cos\theta\langle\xi,e\rangle$. We recall the following even functions on $\mathcal{C}_\theta$.
\begin{definition}[see \cite{Mei-Wang-Weng-Convex_capillary_hypersurfaces_of_prescribed_curvature_problem}]
Let $f\in C^2(\mathcal{C}_\theta)$ and $\xi=(\xi_1,\cdots,\xi_n,\xi_{n+1})\in\mathcal{C}_\theta$, we denote $\widehat{\xi}=(-\xi_1,\cdots,-\xi_n,\xi_{n+1})$. If $f(\xi)=f(\widehat{\xi})$, then we call $f$ a even function on $\mathcal{C}_\theta$. A capillary convex body $\widehat{\Sigma}\in\mathcal{K}_\theta^\circ$ is called symmetric if its support function is a even function on $\mathcal{C}_\theta$.
\end{definition}

For the geometric background of the capillary Orlicz-Minkowski problem, please see Section \ref{Preliminaries}. Here, we let $\phi:[0,+\infty)\rightarrow[0,+\infty)$ be a $C^2$-smooth convex function, then the capillary Orlicz-Minkowski problem is actually equivalent to solve the following Robin boundary value problem of the Monge-Amp\`{e}re type equation:
\begin{equation}\label{Robin-problem-of-the-Monge-Ampere-equation}
\left\{
\begin{aligned}
\phi\left(\frac{\ell}{h}\right)h\det(h_{ij}+h\delta_{ij})&=f,\ \ \ \ \ \ \ \ \ \mbox{in}\ \mathcal{C}_\theta,\\
\nabla_\mu h&=\cot\theta h,\ \ \ \mbox{on}\ \partial\mathcal{C}_\theta,
\end{aligned}
\right.
\end{equation}
where $\mu$ is the outward co-normal of $\partial\Sigma$ in $\Sigma$. To ensure the existence and uniqueness of solutions to \eqref{Robin-problem-of-the-Monge-Ampere-equation}, we need some assumptions about $\phi$. Let $\mathcal{O}$ be the class of $C^2$-smooth, strictly increasing, convex, log-concave functions $\phi$ on $[0,+\infty)$ with $\phi(0)=0$, which satisfies
\begin{equation}\label{Assumptions-phi}
\left\{
\begin{aligned}
&A_1: \lim_{x\to0^+}\phi'(x)=0;\\
&A_2: \liminf_{x\to+\infty}\frac{\phi(x)}{x^{n+1}}>0;\\
&A_3: \frac{d}{dx}\log\frac{\phi(x)}{x}\geqslant0,\ \mbox{for all}\ x>0.
\end{aligned}
\right.
\end{equation}
If $\phi(x)=x^p$ for $p\geqslant n+1$, then the assumptions $A_1-A_3$ in \eqref{Assumptions-phi} all hold. Thus, the solutions to \eqref{Robin-problem-of-the-Monge-Ampere-equation} can resolve the capillary $L_p$-Minkowski problem with supercritical exponents in \cite{Mei-Wang-Weng-The_capillary_L_p_Minkowski_problem}.

\begin{definition}
Let $\phi\in\mathcal{O}$. For any $h,f\in C^2(\mathcal{C}_\theta)$, we say that the function $h$ satisfies the orthogonality condition with respect to $f$ if for all $v\in C^2(\mathcal{C}_\theta)$ such that
\begin{equation}\label{Condition-open-proof}
\int_{\mathcal{C}_\theta}\left(\frac{\ell\phi'\left(\frac{\ell}{h}\right)}{h\phi\left(\frac{\ell}{h}\right)}-n-1\right)
\frac{fv}{h\phi\left(\frac{\ell}{h}\right)}=0,
\end{equation}
there holds
\begin{equation*}
\int_{\mathcal{C}_\theta}hv=0.
\end{equation*}
\end{definition}
It is not hard to check that if $\phi(x)=x^p$, then $v=0$ in \eqref{Condition-open-proof}, i.e., the orthogonality condition is trivial in the $L_p$ case. Then, the main result in this paper is stated below.
\begin{theorem}\label{Main-Theorem}
Given $\phi\in\mathcal{O}$ and $\theta\in(0,\frac{\pi}{2})$. Let $f\in C^2(\mathcal{C}_\theta)$ be a even positive function, and satisfy
\begin{equation}\label{Condition-uniqueness}
\frac{1}{n+1}\int_{\mathcal{C}_\theta}f\geqslant\phi(|\widehat{\mathcal{C}_\theta}|^\frac{1}{n+1}).
\end{equation}
There exists a smooth, symmetric $\widehat{\Sigma}\in\mathcal{K}_\theta^\circ$ with $|\widehat{\Sigma}|=1$, such that its support function $h$ solves \eqref{Robin-problem-of-the-Monge-Ampere-equation} and satisfies the orthogonality condition with respect to $f$. Moreover, if equality holds in \eqref{Condition-uniqueness} and $\phi$ is strictly convex, then $\Sigma$ is the spherical cap $|\widehat{\mathcal{C}_\theta}|^{-\frac{1}{n+1}}\mathcal{C}_\theta$.
\end{theorem}
As a consequence, the volume-normalized solutions to the capillary even $L_p$-Minkowski problem with supercritical exponents ($p\geqslant n+1$) have also been obtained. In particular, we can conclude that the constant $\gamma$ in the capillary $L_{n+1}$-Minkowski problem \cite{Mei-Wang-Weng-The_capillary_L_p_Minkowski_problem} for the volume-normalized solutions is 1.

\vskip 2mm

The organization of the paper: In Section \ref{Preliminaries}, we collect some materials concerning capillary convex bodies. The Orlicz-Brunn-Minkowski inequality and Orlicz-Minkowski inequality of capillary convex bodies will be given in Section \ref{Capillary-Orlicz-Brunn-Minkowski-theory}. The a priori estimates for the capillary even Orlicz-Minkowski problem will be provided in Section \ref{A priori estimates}. In Section \ref{The proof of Theorem}, we show the proof of Theorem \ref{Main-Theorem}.

\section{Preliminaries}\label{Preliminaries}

In this section, we provide some backgrounds of capillary convex bodies. Regarding more details about this, one can refer to \cite{Mei-Wang-Weng-The_capillary_Minkowski_problem,Mei-Wang-Weng-Xia-Alexandrov_Fenchel_inequalities_II}. For a capillary convex body $\widehat{\Sigma}\in\mathcal{K}_\theta$, we denote $\widehat{\partial\Sigma}=\partial(\widehat{\Sigma})\setminus\Sigma$. Let $\nu$ be the unit outward normal of $\Sigma$, the contact angle $\theta$ is defined by
\begin{equation*}
\cos(\pi-\theta)=\langle\nu,e\rangle,\ \mbox{along}\ \partial\Sigma.
\end{equation*}
The spherical cap with radius $r$ refers to
\begin{equation*}
\mathcal{C}_{\theta,r}=\{\xi\in\overline{\mathbb{R}^{n+1}_+}\,:\,|\xi-r\cos\theta e|=r\}.
\end{equation*}
In particular, $\mathcal{C}_\theta=\mathcal{C}_{\theta,1}$. Let $\mu$ be the outward co-normal of $\partial\Sigma$ in $\Sigma$.

Let $\{E_1,\cdots,E_n,E_{n+1}=-e\}$ be the standard basis in $\mathbb{R}^{n+1}$, $\widehat{\Sigma}_1,\widehat{\Sigma}_2\in\mathcal{K}_\theta$ are called horizontally homothetic if
\begin{equation*}
\widehat{\Sigma}_1=r\widehat{\Sigma}_2+x
\end{equation*}
for some $r>0$ and $x\in\text{span}\{E_1,\cdots,E_n\}$. Recall that the capillary Gauss map $\widetilde{\nu}:\Sigma\to\mathcal{C}_\theta,\,X\mapsto\nu(X)+\cos\theta e$ is a diffeomorphism map. By the parametrization of the inverse capillary Gauss map $\widetilde{\nu}^{-1}$, the support function of $\widehat{\Sigma}$ is defined by
\begin{equation*}
h_{\widehat{\Sigma}}(\xi)=\langle X,\nu\rangle=\langle\widetilde{\nu}^{-1}(\xi),\xi-\cos\theta e\rangle, \ \xi\in\mathcal{C}_\theta.
\end{equation*}
In particular, the support function of $\mathcal{C}_\theta$ is
\begin{equation*}
\ell(\xi)=\langle\xi,\xi-\cos\theta e\rangle=|\xi|^2-\cos\theta\langle\xi,e\rangle.
\end{equation*}
For $\xi\in\mathcal{C}_\theta$, we have $|\xi-\cos\theta e|=1$. Squaring it to get $|\xi|^2-2\cos\theta\langle\xi,e\rangle+\cos^2\theta=1$. Thus, we have
\begin{equation*}
\ell(\xi)=|\xi|^2-\cos\theta\langle\xi,e\rangle=\sin^2\theta+\cos\theta\langle\xi,e\rangle.
\end{equation*}

Following \cite{Mei-Wang-Weng-The_capillary_Minkowski_problem}, the capillary support function of $\widehat{\Sigma}$ is defined by
\begin{equation*}
u_{\widehat{\Sigma}}(\xi)=\frac{h_{\widehat{\Sigma}}(\xi)}{\ell(\xi)}, \ \xi\in\mathcal{C}_\theta,
\end{equation*}
and there hold on $\partial\mathcal{C}_\theta$
\begin{equation}\label{Preliminaries-formula-1}
\left\{
\begin{aligned}
&\nabla_\mu h_{\widehat{\Sigma}}=\cot\theta h_{\widehat{\Sigma}},\\
&\nabla_\mu u_{\widehat{\Sigma}}=0.
\end{aligned}
\right.
\end{equation}
In particular, the capillary support function of $\mathcal{C}_\theta$ is $u_{\mathcal{C}_\theta}(\xi)=1$. Along $\partial\mathcal{C}_\theta$, we choose an orthonormal frame $\{e_i\}_{i=1}^n$ with $e_n=\mu$. Then, Proposition 2.8 in \cite{Mei-Wang-Weng-The_capillary_Minkowski_problem} shows on $\partial\mathcal{C}_\theta$
\begin{equation}\label{Preliminaries-formula-2}
\left\{
\begin{aligned}
&h_{in}=0,\\
&u_{in}=-\cot\theta u_i,\ i=1,\cdots,n-1.
\end{aligned}
\right.
\end{equation}

\section{Capillary Orlicz-Brunn-Minkowski theory}\label{Capillary-Orlicz-Brunn-Minkowski-theory}

In this section, we provide a basic framework of the Orlicz-Brunn-Minkowski theroy for capillary convex bodies. Firstly, we introduce the capillary Orlicz combination of capillary convex bodies. Then, we establish the capillary Orlicz-Brunn-Minkowski inequality and the capillary Orlicz-Minkowski inequality. Finally, we propose the capillary Orlicz-Minkowski problem.

\subsection{Capillary Orlicz combination}

Now, we develop the Orlicz addition of Gardner, Hug and Weil \cite{Gardner-Hug-Weil-The_Orlicz_Brunn_Minkowski_theory} (also see e.g., \cite{Xi-Jin-Leng-The_Orlicz_Brunn_Minkowski_inequality}) to the capillary setting.
\begin{definition}\label{Definition-capillary-Orlicz-combination}
Given $\phi\in\mathcal{O}$ and $\alpha,\beta\geqslant0$ with $\alpha^2+\beta^2>0$. For $\widehat{\Sigma}_1,\widehat{\Sigma}_2\in\mathcal{K}_\theta^\circ$, we define the capillary Orlicz combination $M_\phi(\alpha,\beta;\widehat{\Sigma}_1,\widehat{\Sigma}_2)$ by
\begin{equation*}
h_{M_\phi(\alpha,\beta;\widehat{\Sigma}_1,\widehat{\Sigma}_2)}(\xi)=\inf\Bigg\{t>0\,:\,\alpha\phi\left(\frac{h_{\widehat{\Sigma}_1}
(\xi)}{t}\right)+\beta\phi\left(\frac{h_{\widehat{\Sigma}_2}(\xi)}{t}\right)\leqslant1\Bigg\},\ \xi\in\mathcal{C}_\theta.
\end{equation*}
\end{definition}
If $\phi(x)=x^p$, $p\geqslant1$, then $M_\phi(\alpha,\beta;\widehat{\Sigma}_1,\widehat{\Sigma}_2)$ is the capillary $L_p$-combination $\alpha\cdot\widehat{\Sigma}_1+_p\beta\cdot\widehat{\Sigma}_2$ in \cite[Definition 2.1]{Mei-Wang-Weng-The_capillary_L_p_Minkowski_problem}. Next, we show that the above definition is still well-defined in the capillary setting.

\begin{lemma}
Given $\phi\in\mathcal{O}$ and $\alpha,\beta\geqslant0$ with $\alpha^2+\beta^2>0$. Let $\widehat{\Sigma}_1,\widehat{\Sigma}_2\in\mathcal{K}_\theta^\circ$, then $M_\phi(\alpha,\beta;\widehat{\Sigma}_1,\widehat{\Sigma}_2)\in\mathcal{K}_\theta^\circ$, i.e., $h_{M_\phi(\alpha,\beta;\widehat{\Sigma}_1,\widehat{\Sigma}_2)}$ is indeed the support function of $M_\phi(\alpha,\beta;\widehat{\Sigma}_1,\widehat{\Sigma}_2)$.
\end{lemma}
\begin{proof}
Since $\phi\in\mathcal{O}$ is strictly increasing on $[0,+\infty)$, we know from \cite{Xi-Jin-Leng-The_Orlicz_Brunn_Minkowski_inequality,Gardner-Hug-Weil-The_Orlicz_Brunn_Minkowski_theory} that $h_{M_\phi(\alpha,\beta;\widehat{\Sigma}_1,\widehat{\Sigma}_2)}$ is spherical convex on $\mathcal{C}_\theta$. Next, we show that $h_{M_\phi(\alpha,\beta;\widehat{\Sigma}_1,\widehat{\Sigma}_2)}$ satisfies the following Robin boundary condition
\begin{equation*}
\nabla_\mu h_{M_\phi(\alpha,\beta;\widehat{\Sigma}_1,\widehat{\Sigma}_2)}=\cot\theta h_{M_\phi(\alpha,\beta;\widehat{\Sigma}_1,\widehat{\Sigma}_2)},\ \mbox{on}\ \partial\mathcal{C}_\theta.
\end{equation*}
Denote $h_{\widehat{\Sigma}_1}=h_1$, $h_{\widehat{\Sigma}_2}=h_2$ and $h_{M_\phi(\alpha,\beta;\widehat{\Sigma}_1,\widehat{\Sigma}_2)}=h$. From Definition \ref{Definition-capillary-Orlicz-combination}, there holds
\begin{equation*}
\alpha\phi\left(\frac{h_1}{h}\right)+\beta\phi\left(\frac{h_2}{h}\right)=1.
\end{equation*}
Taking the differential of the above formula yields
\begin{equation*}
\alpha\phi'\left(\frac{h_1}{h}\right)(\nabla_\mu h_1\,h-h_1\nabla_\mu h)+\beta\phi'\left(\frac{h_2}{h}\right)(\nabla_\mu h_2\,h-h_2\nabla_\mu h)=0.
\end{equation*}
Substituting $\nabla_\mu h_1=\cot\theta h_1$ and $\nabla_\mu h_2=\cot\theta h_2$ into the above formula gives
\begin{equation*}
\Bigg(\alpha h_1\phi'\left(\frac{h_1}{h}\right)+\beta h_2\phi'\left(\frac{h_2}{h}\right)\Bigg)(\cot\theta h-\nabla_\mu h)=0.
\end{equation*}
Using the assumption $A_3$ in \eqref{Assumptions-phi}, we have
\begin{equation*}
\phi'(x)\geqslant\frac{\phi(x)}{x}>0,\ \mbox{for all}\ x>0.
\end{equation*}
Thus, $\alpha h_1\phi'(h_1/h)+\beta h_2\phi'(h_2/h)>0$, which shows $\nabla_\mu h=\cot\theta h$. Finally, by
\cite[Proposition 2.6]{Mei-Wang-Weng-Xia-Alexandrov_Fenchel_inequalities_II}, we know that $h_{M_\phi(\alpha,\beta;\widehat{\Sigma}_1,\widehat{\Sigma}_2)}$ is the support function of capillary convex body $M_\phi(\alpha,\beta;\widehat{\Sigma}_1,\widehat{\Sigma}_2)\in\mathcal{K}_\theta^\circ$.
\end{proof}

\subsection{Capillary Orlicz-Brunn-Minkowski inequality}

The capillary analogues of the Orlicz-Brunn-Minkowski inequality and Orlicz-Minkowski inequality are given in this subsection. For the proof, our main tool is the Aleksandrov-Fenchel inequality for capillary convex hypersurfaces. In some special cases, this inequality was firstly proved by Wang, Weng and Xia \cite{Wang-Weng-Xia-Alexandrov_Fenchel_inequalities_I} by the locally constrained inverse mean curvature flows. Later, there are many important developments concerning this inequality, please see \cite{Hu-Wei-Yang-Zhou-A_complete_family_of_Alexandrov_Fenchel_inequalities,
Mei-Wang-Weng-A_constrained_mean_curvature_flow_and_Alexandrov_Fenchel_inequalities,
Mei-Weng-A_fully_nonlinear_locally_constrained_curvature_flow_for_capillary_hypersurface,
Wang-Weng-Xia-A_Minkowski_type_inequality_for_capillary_hypersurfaces_in_a_half_space}. Finally, by the spectral methods of elliptic differential operators of Shenfeld and van Handel \cite{Shenfeld-van-Handel-Mixed_volumes_and_the_Bochner_method,Shenfeld-van-Handel-The_extremals_of_Minkowski_quadratic_inequality,
Shenfeld-van-Handel-The_extremals_of_the_Alexandrov_Fenchel_inequality_for_convex_polytopes}, Mei, Wang, Weng and Xia \cite{Mei-Wang-Weng-Xia-Alexandrov_Fenchel_inequalities_II} established the Aleksandrov-Fenchel inequality for all $\theta\in(0,\pi)$.

Denote by $\delta_{\mathbb{S}^n}$ the standard round metric on the unit sphere $\mathbb{S}^n$, and by $\nabla$ the Levi-Civita connection of $\delta_{\mathbb{S}^n}$. For $f\in C^2(\mathcal{C}_\theta)$, the area operator is defined by
\begin{equation*}
A[f]=\nabla^2f+f\delta_{\mathbb{S}^n}.
\end{equation*}
Recall that the mixed discriminant $Q:(\mathbb{R}^{n\times n})^n\to\mathbb{R}$ is defined by
\begin{equation*}
\det(\lambda_1A_1+\cdots+\lambda_mA_m)=\sum_{i_1,\cdots,i_n=1}^m\lambda_{i_1}\cdots\lambda_{i_n}Q(A_{i_1},\cdots,A_{i_n})
\end{equation*}
for $m\in\mathbb{N}$, $\lambda_1,\cdots,\lambda_m\geqslant0$ and the symmetric matrices $A_1,\cdots,A_m\in\mathbb{R}^{n\times n}$. Then, for $f,f_1,\cdots,f_n\in C^2(\mathcal{C}_\theta)$, the authors in \cite{Mei-Wang-Weng-Xia-Alexandrov_Fenchel_inequalities_II} introduced the mixed volume of $f,f_1,\cdots,f_n$ as follows
\begin{equation*}
V(f,f_1,\cdots,f_n)=\frac{1}{n+1}\int_{\mathcal{C}_\theta}fQ(A[f_1],\cdots,A[f_n])\,d\xi.
\end{equation*}
And the mixed volume of $\widehat{\Sigma},\widehat{\Sigma}_1,\cdots,\widehat{\Sigma}_n\in\mathcal{K}_\theta$ is defined by
\begin{equation*}
V(\widehat{\Sigma},\widehat{\Sigma}_1,\cdots,\widehat{\Sigma}_n)=V(h_{\widehat{\Sigma}},h_{\widehat{\Sigma}_1},\cdots,
h_{\widehat{\Sigma}_n}).
\end{equation*}
In particular, we denote
\begin{equation*}
V_1(\widehat{\Sigma}_1,\widehat{\Sigma}_2)=V(\overbrace{\widehat{\Sigma}_1,\cdots,\widehat{\Sigma}_1}^n,\widehat{\Sigma}_2),\
V_1(\widehat{\Sigma}_1,f)=V(\overbrace{h_{\widehat{\Sigma}_1},\cdots,h_{\widehat{\Sigma}_n}}^n,f).
\end{equation*}

\begin{theorem}[see \cite{Mei-Wang-Weng-Xia-Alexandrov_Fenchel_inequalities_II}]
Let $\theta\in(0,\pi)$ and $\widehat{\Sigma},\widehat{\Sigma}_1,\cdots,\widehat{\Sigma}_n\in\mathcal{K}_\theta$, there holds the Aleksandrov-Fenchel inequality
\begin{equation}\label{The-capillary-Aleksandrov-Fenchel-inequality}
V^2(\widehat{\Sigma},\widehat{\Sigma}_1,\widehat{\Sigma}_2,\cdots,\widehat{\Sigma}_n)\geqslant
V(\widehat{\Sigma},\widehat{\Sigma},\widehat{\Sigma}_2,\cdots,\widehat{\Sigma}_n)
V(\widehat{\Sigma}_1,\widehat{\Sigma}_1,\widehat{\Sigma}_2,\cdots,\widehat{\Sigma}_n),
\end{equation}
with equality if and only if $\widehat{\Sigma}$ and $\widehat{\Sigma}_1$ are horizontally homothetic.
\end{theorem}

Given $\phi\in\mathcal{O}$ and $\varepsilon>0$, we denote the Orlicz perturbation $\widehat{\Sigma}_\varepsilon=\widehat{\Sigma}_1+_\phi\varepsilon\widehat{\Sigma}_2:=M_\phi(1,\varepsilon;\widehat{\Sigma}_1,\widehat{\Sigma}_2)$ for $\widehat{\Sigma}_1,\widehat{\Sigma}_2\in\mathcal{K}_\theta^\circ$. By
\cite[Lemma 8.4]{Gardner-Hug-Weil-The_Orlicz_Brunn_Minkowski_theory}, there holds
\begin{equation*}
\frac{d}{d\varepsilon}\bigg|_{\varepsilon=0}h_{\widehat{\Sigma}_\varepsilon}=\frac{1}{\phi'(1)}h_{\widehat{\Sigma}_1}
\phi\left(\frac{h_{\widehat{\Sigma}_2}}{h_{\widehat{\Sigma}_1}}\right).
\end{equation*}
Then, we have
\begin{equation}\label{Variational-Formula}
\frac{d}{d\varepsilon}\bigg|_{\varepsilon=0}V(\widehat{\Sigma}_\varepsilon)=\frac{1}{\phi'(1)}\int_{\mathcal{C}_\theta}
\phi\left(\frac{h_{\widehat{\Sigma}_2}}{h_{\widehat{\Sigma}_1}}\right)h_{\widehat{\Sigma}_1}
\det((h_{\widehat{\Sigma}_1})_{ij}+h_{\widehat{\Sigma}_1}\delta_{ij})\,d\xi.
\end{equation}
Therefore, we define the Orlicz mixed volume of capillary convex bodies as follows.
\begin{definition}
Let $\phi\in\mathcal{O}$ and $\widehat{\Sigma}_1,\widehat{\Sigma}_2\in\mathcal{K}_\theta^\circ$. The Orlicz mixed volume of $\widehat{\Sigma}_1$ and $\widehat{\Sigma}_2$ is defined by
\begin{equation*}
V_\phi(\widehat{\Sigma}_1,\widehat{\Sigma}_2)=\frac{1}{n+1}\int_{\mathcal{C}_\theta}\phi\bigg(\frac{h_{\widehat{\Sigma}_2}}
{h_{\widehat{\Sigma}_1}}\bigg)h_{\widehat{\Sigma}_1}\det((h_{\widehat{\Sigma}_1})_{ij}+h_{\widehat{\Sigma}_1}\delta_{ij})\,d\xi.
\end{equation*}
\end{definition}
\begin{remark}
If we choose $\phi(x)=x^p$, then $V_\phi(\widehat{\Sigma}_1,\ell^\frac{1}{p}h_{\widehat{\Sigma}_2})=V^c_p(\widehat{\Sigma}_1,\widehat{\Sigma}_2)$. Here, the capillary $L_p$ mixed volumes \cite{Mei-Wang-Weng-The_capillary_L_p_Minkowski_problem} are defined by
\begin{equation*}
V^c_p(\widehat{\Sigma}_1,\widehat{\Sigma}_2)=\frac{1}{n+1}\int_{\mathcal{C}_\theta}\ell h^p_{\widehat{\Sigma}_2}
h^{1-p}_{\widehat{\Sigma}_1}\det((h_{\widehat{\Sigma}_1})_{ij}+h_{\widehat{\Sigma}_1}\delta_{ij})\,d\xi.
\end{equation*}
\end{remark}

For convenience, we denote the capillary cone-volume measure
\begin{equation*}
dV^c(\widehat{\Sigma},\cdot)=\frac{1}{n+1}\ell\,h_{\widehat{\Sigma}}\det((h_{\widehat{\Sigma}})_{ij}+h_{\widehat{\Sigma}}\delta_{ij})\,d\xi,
\end{equation*}
and the capillary Orlicz volume
\begin{equation*}
V_\phi(\widehat{\Sigma})=V_\phi(\widehat{\Sigma},\widehat{\Sigma}).
\end{equation*}
From the Aleksandrov-Fenchel inequality \eqref{The-capillary-Aleksandrov-Fenchel-inequality} and the Jensen inequality, we can establish the following Orlicz-Minkowski inequality.
\begin{theorem}
Let $\phi\in\mathcal{O}$ and $\widehat{\Sigma}_1,\widehat{\Sigma}_2\in\mathcal{K}_\theta^\circ$. There holds the Orlicz-Minkowski inequality
\begin{equation}\label{Orlicz-Minkowski-inequality}
V_\phi(\widehat{\Sigma}_1,\widehat{\Sigma}_2)\geqslant V(\widehat{\Sigma}_1)\,\phi\left(\frac{V(\widehat{\Sigma}_2)^\frac{1}{n+1}}{V(\widehat{\Sigma}_1)^\frac{1}{n+1}}\right),
\end{equation}
with equality if $\widehat{\Sigma}_1$ and $\widehat{\Sigma}_2$ are dilates. When $\phi$ is strictly convex, equality holds if and only if $\widehat{\Sigma}_1$ and $\widehat{\Sigma}_2$ are dilates.
\end{theorem}
\begin{proof}
Note that
\begin{equation*}
\int_{\mathcal{C}_\theta}\frac{1}{\ell(\xi)}\,dV^c(\widehat{\Sigma}_1,\xi)=V(\widehat{\Sigma}_1).
\end{equation*}
By the Jensen inequality, we have
\begin{equation*}
V_\phi(\widehat{\Sigma}_1,\widehat{\Sigma}_2)\geqslant V(\widehat{\Sigma}_1)\phi\bigg(\frac{1}{V(\widehat{\Sigma}_1)}
\int_{\mathcal{C}_\theta}\frac{1}{\ell(\xi)}\frac{h_{\widehat{\Sigma}_2}(\xi)}{h_{\widehat{\Sigma}_1}(\xi)}\,dV^c(\widehat{\Sigma}_1,\xi)
\bigg)=V(\widehat{\Sigma}_1)\phi\bigg(\frac{V_1(\widehat{\Sigma}_1,\widehat{\Sigma}_2)}{V(\widehat{\Sigma}_1)}\bigg).
\end{equation*}
From the Aleksandrov-Fenchel inequality \eqref{The-capillary-Aleksandrov-Fenchel-inequality}, we can derive the following Minkowski inequality (see \cite{Schneider-book})
\begin{equation*}
V_1(\widehat{\Sigma}_1,\widehat{\Sigma}_2)^{n+1}\geqslant V(\widehat{\Sigma}_1)^nV(\widehat{\Sigma}_2),
\end{equation*}
where equality holds if and only if $\widehat{\Sigma}_1$ and $\widehat{\Sigma}_2$ are horizontally homothetic. Then, we obtain
\begin{equation*}
\frac{V_1(\widehat{\Sigma}_1,\widehat{\Sigma}_2)}{V(\widehat{\Sigma}_1)}\geqslant\frac{V(\widehat{\Sigma}_2)^\frac{1}{n+1}}
{V(\widehat{\Sigma}_1)^\frac{1}{n+1}}.
\end{equation*}
Using the monotonicity of $\phi$, the desired inequality is obtained. Regarding the equality condition, it is easy to check from the Aleksandrov-Fenchel inequality and the Jensen inequality.
\end{proof}

Using the above Orlicz-Minkowski inequality, we can establish the following Orlicz-Brunn-Minkowski inequality.
\begin{theorem}
Let $\phi\in\mathcal{O}$, $\alpha,\beta\geqslant0$, $\alpha^2+\beta^2>0$, $\widehat{\Sigma}_1,\widehat{\Sigma}_2\in\mathcal{K}_\theta^\circ$. There holds the Orlicz Brunn-Minkowski inequality
\begin{equation}\label{Orlicz-Brunn-Minkowski-inequality}
\alpha\phi\left(\frac{V(\widehat{\Sigma}_1)^\frac{1}{n+1}}{V(M_\phi(\alpha,\beta;\widehat{\Sigma}_1,\widehat{\Sigma}_2))^\frac{1}{n+1}}\right)
+\beta\phi\left(\frac{V(\widehat{\Sigma}_2)^\frac{1}{n+1}}{V(M_\phi(\alpha,\beta;\widehat{\Sigma}_1,\widehat{\Sigma}_2))^\frac{1}{n+1}}\right)
\leqslant1,
\end{equation}
with equality if $\widehat{\Sigma}_1$ and $\widehat{\Sigma}_2$ are dilates. When $\phi$ is strictly convex, equality holds if and only if $\widehat{\Sigma}_1$ and $\widehat{\Sigma}_2$ are dilates.
\end{theorem}
\begin{proof}
We denote $\widehat{\Sigma}=M_\phi(\alpha,\beta;\widehat{\Sigma}_1,\widehat{\Sigma}_2)$, and denote by $h,h_1,h_2$ the support functions of $\widehat{\Sigma},\widehat{\Sigma}_1,\widehat{\Sigma}_2$, respectively. Using Proposition \ref{Orlicz-Minkowski-inequality}, we have
\begin{equation*}
\begin{aligned}
V_\phi(\widehat{\Sigma})&=V_\phi(\widehat{\Sigma},\widehat{\Sigma})\\
&=\phi(1)\int_{\mathcal{C}_\theta}\frac{1}{\ell(\xi)}\left[\alpha\phi\left(\frac{h_1}{h}\right)+\beta\phi\left(\frac{h_2}{h}\right)
\right]\,dV^c(\widehat{\Sigma},\xi)\\
&=\phi(1)\left(\alpha V_\phi(\widehat{\Sigma},\widehat{\Sigma}_1)+\beta V_\phi(\widehat{\Sigma},\widehat{\Sigma}_2)\right)\\
&\geqslant\phi(1)\left(\alpha V(\widehat{\Sigma})\,\phi\left(\frac{V(\widehat{\Sigma}_1)^\frac{1}{n+1}}{V(\widehat{\Sigma})^\frac{1}{n+1}}\right)+\beta V(\widehat{\Sigma})\,\phi\left(\frac{V(\widehat{\Sigma}_2)^\frac{1}{n+1}}{V(\widehat{\Sigma})^\frac{1}{n+1}}\right)\right)\\
&=V_\phi(\widehat{\Sigma})\left(\alpha\phi\left(\frac{V(\widehat{\Sigma}_1)^\frac{1}{n+1}}{V(\widehat{\Sigma})^\frac{1}{n+1}}\right)
+\beta\phi\left(\frac{V(\widehat{\Sigma}_2)^\frac{1}{n+1}}{V(\widehat{\Sigma})^\frac{1}{n+1}}\right)\right),
\end{aligned}
\end{equation*}
this completes the proof of the theorem.
\end{proof}

\begin{remark}
From the proof of the above theorems, we conclude that if we remove conditions $A_1$ and $A_2$ in \eqref{Assumptions-phi}, then the above Orlicz-Minkowski inequality and Orlicz-Brunn-Minkowski inequality are still hold for all $\theta\in(0,\pi)$.
\end{remark}

We now derive the equivalence between the Orlicz-Minkowski inequality \eqref{Orlicz-Minkowski-inequality} and the Orlicz-Brunn-Minkowski inequality \eqref{Orlicz-Brunn-Minkowski-inequality}. We have proved the Orlicz-Brunn-Minkowski inequality \eqref{Orlicz-Brunn-Minkowski-inequality} by the Orlicz-Minkowski inequality \eqref{Orlicz-Minkowski-inequality}. Thus, we only need to prove the Orlicz-Minkowski inequality \eqref{Orlicz-Minkowski-inequality} by the Orlicz-Brunn-Minkowski inequality \eqref{Orlicz-Brunn-Minkowski-inequality}.
\begin{proof}[Proof (\eqref{Orlicz-Brunn-Minkowski-inequality}$\Rightarrow$ \eqref{Orlicz-Minkowski-inequality})]
Given $\phi\in\mathcal{O}$ with $\phi(1)=1$. For $\widehat{\Sigma}_1,\widehat{\Sigma}_2\in\mathcal{K}_\theta^\circ$, we denote \begin{equation*}
\widehat{\Sigma}_\varepsilon=\widehat{\Sigma}_1+_\phi\varepsilon\widehat{\Sigma}_2,
\end{equation*}
for $\varepsilon>0$. Define the function
\begin{equation*}
f(\varepsilon)=\phi\left(\frac{V(\widehat{\Sigma}_1)^\frac{1}{n+1}}{V(\widehat{\Sigma}_\varepsilon)^\frac{1}{n+1}}\right)
+\varepsilon\phi\left(\frac{V(\widehat{\Sigma}_2)^\frac{1}{n+1}}{V(\widehat{\Sigma}_\varepsilon)^\frac{1}{n+1}}\right)-1,\ \varepsilon>0.
\end{equation*}
By the Orlicz-Brunn-Minkowski inequality \eqref{Orlicz-Brunn-Minkowski-inequality}, we know that $f$ is non-positive and convex on $(0,+\infty)$. Thus, we have
\begin{equation*}
\begin{aligned}
0&\geqslant\frac{d}{d\varepsilon}\bigg|_{\varepsilon=0^+}f(\varepsilon)=\lim_{\varepsilon\to0^+}\frac{f(\varepsilon)-f(0)}{\varepsilon}\\
&=\lim_{\varepsilon\to0^+}\frac{\phi\left(\frac{V(\widehat{\Sigma}_1)^\frac{1}{n+1}}{V(\widehat{\Sigma}_\varepsilon)^\frac{1}{n+1}}
\right)-1}{\varepsilon}+\lim_{\varepsilon\to0^+}\phi\left(\frac{V(\widehat{\Sigma}_2)^\frac{1}{n+1}}
{V(\widehat{\Sigma}_\varepsilon)^\frac{1}{n+1}}\right)\\
&=-\frac{1}{n+1}\phi'(1)V(\widehat{\Sigma}_1)^{-1}\frac{d}{d\varepsilon}\bigg|_{\varepsilon=0^+}V(\widehat{\Sigma}_\varepsilon)
+\phi\left(\lim_{\varepsilon\to0^+}\frac{V(\widehat{\Sigma}_2)^\frac{1}{n+1}}{V(\widehat{\Sigma}_\varepsilon)^\frac{1}{n+1}}\right)\\
&=-V(\widehat{\Sigma}_1)^{-1}V_\phi(\widehat{\Sigma}_1,\widehat{\Sigma}_2)+\phi\left(\frac{V(\widehat{\Sigma}_2)^\frac{1}{n+1}}
{V(\widehat{\Sigma}_1)^\frac{1}{n+1}}\right),
\end{aligned}
\end{equation*}
where we used the L'Hospital's rule and the variational formula \eqref{Variational-Formula}. Thus, we obtain the Orlicz-Minkowski inequality \eqref{Orlicz-Minkowski-inequality}. If the equality holds in \eqref{Orlicz-Minkowski-inequality}, then
\begin{equation*}
\frac{d}{d\varepsilon}\bigg|_{\varepsilon=0^+}f(\varepsilon)=0.
\end{equation*}
This reduces $f\equiv0$, i.e., the equality holds in Orlicz-Brunn-Minkowski inequality \eqref{Orlicz-Brunn-Minkowski-inequality}, so $\widehat{\Sigma}_1$ and $\widehat{\Sigma}_2$ are dilates whenever $\phi$ is strictly convex.
\end{proof}

\subsection{Capillary Orlicz surface area measure}

Let $\widehat{\Sigma}\in\mathcal{K}_\theta$, the wetting energy of $\widehat{\Sigma}$ is given by
\begin{equation*}
A(\widehat{\Sigma})=\int_{\mathcal{C}_\theta}\ell(\xi)\det((h_{\widehat{\Sigma}})_{ij}+h_{\widehat{\Sigma}}\delta_{ij})\,d\xi,
\end{equation*}
and the capillary surface area measure \cite{Mei-Wang-Weng-The_capillary_Minkowski_problem} of $\widehat{\Sigma}$ is defined by \begin{equation*}
dS^c(\widehat{\Sigma},\xi)=\ell(\xi)\det((h_{\widehat{\Sigma}})_{ij}+h_{\widehat{\Sigma}}\delta_{ij})\,d\xi.
\end{equation*}
Then, we can find that $S^c(\widehat{\Sigma},\cdot)$ is a localization of the wetting energy $A(\widehat{\Sigma})$. Inspired by the variational formula \eqref{Variational-Formula}, we can find that the variation $\frac{d}{d\varepsilon}\big|_{\varepsilon=0}V(\widehat{\Sigma}_\varepsilon)$ is a Orlicz surface area of $\widehat{\Sigma}_1$ for $\widehat{\Sigma}_2=\widehat{\mathcal{C}_\theta}$. Therefore, we define the capillary Orlicz surface area measure as follows.
\begin{definition}
Given $\phi\in\mathcal{O}$. We define the capillary Orlicz surface area measure of $\widehat{\Sigma}\in\mathcal{K}_\theta^\circ$ by
\begin{equation*}
S^c_\phi(\widehat{\Sigma},\omega)=\int_\omega\phi\bigg(\frac{\ell(\xi)}{h_{\widehat{\Sigma}}(\xi)}\bigg) h_{\widehat{\Sigma}}(\xi)\det((h_{\widehat{\Sigma}})_{ij}+h_{\widehat{\Sigma}}\delta_{ij})\,d\xi,
\end{equation*}
for any Borel set $\omega\subset\mathcal{C}_\theta$.
\end{definition}

If we choose $\phi(x)=x^p$, then the capillary Orlicz surface area measure $S^c_\phi(\widehat{\Sigma},\cdot)$ reduces to the capillary $L_p$ surface area measure $S^c_p(\widehat{\Sigma},\cdot)$ \cite{Mei-Wang-Weng-The_capillary_L_p_Minkowski_problem} without considering the power of $\ell$. Here, the capillary $L_p$ surface area measure of $\widehat{\Sigma}$ is defined by
\begin{equation*}
dS^c_p(\widehat{\Sigma},\xi)=\ell(\xi)h^{1-p}_{\widehat{\Sigma}}(\xi)\det((h_{\widehat{\Sigma}})_{ij}+h_{\widehat{\Sigma}}\delta_{ij})
\,d\xi,\ \xi\in\mathcal{C}_\theta.
\end{equation*}
Naturally, we propose a capillary version of the Orlicz-Minkowski problem as follows.
\vskip 2mm
\noindent {\bf Capillary Orlicz-Minkowski problem:} {\it Let $\phi\in\mathcal{O}$. Given a positive, smooth function $f$ on $\mathcal{C}_\theta$, what does there exist a capillary convex body $\widehat{\Sigma}\in\mathcal{K}_\theta^\circ$ such that
\begin{equation*}
\frac{dS^c_\phi(\widehat{\Sigma},\xi)}{d\xi}=\phi\bigg(\frac{\ell(\xi)}{h_{\widehat{\Sigma}}(\xi)}\bigg) h_{\widehat{\Sigma}}(\xi)\det((h_{\widehat{\Sigma}})_{ij}+h_{\widehat{\Sigma}}\delta_{ij})=f(\xi),\ \ \forall\xi\in\mathcal{C}_\theta?
\end{equation*}}

Using \eqref{Preliminaries-formula-1}, we can reduce the capillary Orlicz-Minkowski problem to the Monge-Amp\`{e}re type equation with a Robin boundary value condition as follows:
\begin{equation*}
\left\{
\begin{aligned}
\phi\left(\frac{\ell}{h}\right)h\det(\nabla^2h+h\delta_{\mathbb{S}^n})&=f,\ \ \ \ \ \ \ \ \ \mbox{in}\ \mathcal{C}_\theta,\\
\nabla_\mu h&=\cot\theta h,\ \ \ \mbox{on}\ \partial\mathcal{C}_\theta.
\end{aligned}
\right.
\end{equation*}
From \cite[Proposition 2.4]{Mei-Wang-Weng-The_capillary_Minkowski_problem}, we know
\begin{equation*}
\det(\nabla^2h+h\delta_{\mathbb{S}^n})=\det(\ell\nabla^2u+\cos\theta(\nabla u\otimes e^T+e^T\otimes\nabla u)+u\delta_{\mathbb{S}^n}).
\end{equation*}
Therefore, the capillary Orlicz-Minkowski problem is also equivalent to a Neumann boundary value problem
\begin{equation*}
\left\{
\begin{aligned}
\phi\left(\frac{1}{u}\right)u\det(\ell\nabla^2u+\cos\theta(\nabla u\otimes e^T+e^T\otimes\nabla u)+u\delta_{\mathbb{S}^n})&=f,\ \ \ \ \ \ \ \ \ \mbox{in}\ \mathcal{C}_\theta,\\
\nabla_\mu u&=0,\ \ \ \ \ \ \ \ \ \mbox{on}\ \partial\mathcal{C}_\theta.
\end{aligned}
\right.
\end{equation*}

\section{A priori estimates}\label{A priori estimates}

To solve the Robin boundary value problem \eqref{Robin-problem-of-the-Monge-Ampere-equation}, we need the a priori estimates of \eqref{Robin-problem-of-the-Monge-Ampere-equation}. For the more wide application scope, we consider the following normalized problem:
\begin{equation}\label{Robin-problem-of-the-Monge-Ampere-equation-normalized}
\left\{
\begin{aligned}
\frac{1}{|\widehat{\Sigma}|}\phi\Big(\frac{\ell}{h}\Big)h\det(h_{ij}+h\delta_{ij})&=f,\ \ \ \ \ \ \ \ \ \mbox{in}\ \mathcal{C}_\theta,\\
\nabla_\mu h&=\cot\theta h,\ \ \ \mbox{on}\ \partial\mathcal{C}_\theta.
\end{aligned}
\right.
\end{equation}
Here $|\widehat{\Sigma}|=V(\widehat{\Sigma})$ is the volume of $\widehat{\Sigma}\in\mathcal{K}_\theta^\circ$. It is easy to check that the a priori estimates of \eqref{Robin-problem-of-the-Monge-Ampere-equation} with $|\widehat{\Sigma}|=1$ can be directly derived from the a priori estimates of \eqref{Robin-problem-of-the-Monge-Ampere-equation-normalized}.

\subsection{$C^0$-estimate}

\begin{lemma}\label{Lemma-C-0-estimate}
Given $\phi\in\mathcal{O}$ and $\theta\in(0,\pi)$. Let $h$ be a positive, capillary even convex solution to \eqref{Robin-problem-of-the-Monge-Ampere-equation-normalized}, then there exists a constant $C_0>0$ depending only on $f,n,\theta,\phi$ such that
\begin{equation}\label{C-0-estimate}
\frac{1}{C_0}\leqslant h\leqslant C_0,\ \ {\rm on}\ \mathcal{C}_\theta.
\end{equation}
\end{lemma}
\begin{proof}
Since the Steiner point of a symmetric convex body is the origin and $h$ is even, we have
\begin{equation*}
\int_{\mathcal{C}_\theta}h(\xi)\langle\xi,E_i\rangle\,d\xi=0,\ 1\leqslant i\leqslant n.
\end{equation*}
Let $\Sigma$ be a capillary hypersurface corresponding to $h$ as its support function, then by
\cite[Lemma 3.1]{Mei-Wang-Weng-The_capillary_Minkowski_problem} we know that $o\in\text{int}(\widehat{\partial\Sigma})$. Let $R$ denote the smallest positive constant such that $\widehat{\Sigma}\subset\widehat{\mathcal{C}_{\theta,R}}$, then there exists $X\in\Sigma\cap\mathcal{C}_{\theta,R}$. Set $X_0=\frac{X}{R}\in\mathcal{C}_\theta$. For any $\xi\in\mathcal{C}_\theta$, we have
\begin{equation*}
h(\xi)=\sup_{Y\in\Sigma}\langle\xi-\cos\theta e,Y\rangle\geqslant\max\{0,\langle\xi-\cos\theta e,X\rangle\}=R\max\{0,\langle\xi-\cos\theta e,X_0\rangle\}.
\end{equation*}
Define a function
\begin{equation*}
\varphi(s,t)=\frac{1}{\phi\big(\frac{t}{s}\big)},\ \forall\ s,t>0,
\end{equation*}
then $\varphi$ is strictly increasing with respect to $s$ if we fix $t$, and $\varphi$ is strictly decreasing with respect to $t$ if we fix $s$. Hence, together with $\ell\leqslant2$, we have
\begin{equation*}
\varphi(h(\xi),\ell(\xi))\geqslant\varphi(R\max\{0,\langle\xi-\cos\theta e,X_0\rangle\},2),\ \xi\in\mathcal{C}_\theta.
\end{equation*}
From the assumption $A_1$ in \eqref{Assumptions-phi}, we have $\lim_{s\to+\infty}\varphi(s,2)/s=+\infty$, so there exists a constant $N>0$ such that $\varphi(s,2)>s$ for all $s>N$. We define
\begin{equation*}
\omega(X_0)=\Big\{\xi\in\mathcal{C}_\theta\,:\,\langle\xi-\cos\theta e,X_0/|X_0|\rangle\geqslant\frac{\sqrt{2}}{2}\Big\}.
\end{equation*}
For $\xi\in\omega(X_0)$, if $R\max\{0,\langle\xi-\cos\theta e,X_0\rangle\}>N$, then
\begin{equation}\label{C-0-estimate-proof-formula-1}
\varphi(h(\xi),\ell(\xi))\geqslant\frac{\sqrt{2}}{2}|X_0|R\geqslant\frac{\sqrt{2}}{2}\min\{1-\cos\theta,\sin\theta\}R.
\end{equation}
If $R\max\{0,\langle\xi-\cos\theta e,X_0\rangle\}\leqslant N$, then
\begin{equation*}
R\leqslant\frac{\sqrt{2}N}{\min\{1-\cos\theta,\sin\theta\}}.
\end{equation*}
Therefore, we can assume that \eqref{C-0-estimate-proof-formula-1} holds on $\omega(X_0)$. Then, integrating \eqref{C-0-estimate-proof-formula-1} on $\omega(X_0)$ yields
\begin{equation}\label{C-0-estimate-proof-formula-2}
\int_{\omega(X_0)}\varphi(h(\xi),\ell(\xi))f(\xi)\,d\xi\geqslant\frac{\sqrt{2}}{2}\min\{1-\cos\theta,\sin\theta\}R\int_{\omega(X_0)}f(\xi)\,d\xi.
\end{equation}
From \eqref{Robin-problem-of-the-Monge-Ampere-equation-normalized}, we have
\begin{equation}\label{C-0-estimate-proof-formula-3}
\int_{\omega(X_0)}\varphi(h(\xi),\ell(\xi))f(\xi)\,d\xi=\frac{1}{|\widehat{\Sigma}|}\int_{\omega(X_0)}h\det(h_{ij}+h\delta_{ij})\,d\xi
\leqslant(n+1).
\end{equation}
Combining \eqref{C-0-estimate-proof-formula-2} with \eqref{C-0-estimate-proof-formula-3}, we get
\begin{equation*}
R\leqslant\frac{(n+1)\sqrt{2}}{\min\{1-\cos\theta,\sin\theta\}\inf\{\int_{\omega(Y)}f(\xi)\,d\xi\,|\,Y\in\mathcal{C}_\theta\}}.
\end{equation*}
Noting that $h\leqslant\max\{\sin\theta,1-\cos\theta\}R$, we can complete the estimate of the upper bound.

By \cite[Proposition 2.4]{Mei-Wang-Weng-The_capillary_Minkowski_problem}, we know that \eqref{Robin-problem-of-the-Monge-Ampere-equation-normalized} is equivalent to
\begin{equation}\label{C-0-estimate-proof-formula-4}
\left\{
\begin{aligned}
\det(\ell\nabla^2u+\cos\theta(\nabla u\otimes e^T+e^T\otimes\nabla u)+u\delta_{\mathbb{S}^n})&=\frac{f|\widehat{\Sigma}|}{\ell u\phi\left(\frac{1}{u}\right)},\ \ \ \ \ \ \mbox{in}\ \mathcal{C}_\theta,\\
\nabla_\mu u&=0,\ \ \ \ \ \ \ \ \ \ \ \ \ \ \ \mbox{on}\ \partial\mathcal{C}_\theta.
\end{aligned}
\right.
\end{equation}
Suppose $u$ attains the minimum value at $\xi_0$. If $\xi_0\in\mathcal{C}_\theta\setminus\partial\mathcal{C}_\theta$, then
\begin{equation}\label{C-0-estimate-proof-formula-5}
\nabla u(\xi_0)=0\ \ \mbox{and}\ \ \nabla^2u(\xi_0)\geqslant0.
\end{equation}
If $\xi_0\in\partial\mathcal{C}_\theta$, $\nabla_\mu u=0$ implies that \eqref{C-0-estimate-proof-formula-5} still holds (If we choose an orthonormal frame $\{e_i\}_{i=1}^n$ at $\xi_0$ with $e_n=\mu$, then $\nabla_iu=0$ for $i=1,\cdots,n-1$ due to the Fermat's lemma). Combining \eqref{C-0-estimate-proof-formula-4} with \eqref{C-0-estimate-proof-formula-5}, we have
\begin{equation}\label{C-0-estimate-proof-formula-6}
|\widehat{\Sigma}|f(\xi_0)\geqslant\ell(\xi_0)u^{n+1}(\xi_0)\phi\left(\frac{1}{u(\xi_0)}\right).
\end{equation}
Assume that there is no lower bound for $h$, then $u(\xi_0)$ can tend to zero. Since $h$ is even and has the upper bound, the volume $|\widehat{\Sigma}|$ can tend to zero. Thus, from \eqref{C-0-estimate-proof-formula-6} we get
\begin{equation*}
u^{n+1}(\xi_0)\phi\left(\frac{1}{u(\xi_0)}\right)\rightarrow0,\ \mbox{if}\ u(\xi_0)\to0,
\end{equation*}
which contradicts with the assumption $A_2$ in \eqref{Assumptions-phi}
\begin{equation*}
\liminf_{x\to+\infty}\frac{\phi(x)}{x^{n+1}}>0.
\end{equation*}
This completes the proof of the lemma.
\end{proof}

\subsection{$C^1$-estimate}

Recall the distance function
\begin{equation}\label{The-distance-function}
d(\xi)={\rm dist}_{\delta_{\mathbb{S}^n}}(\xi,\partial\mathcal{C}_\theta),
\end{equation}
which is smooth except at the north pole and satisfies $\nabla d=-\mu$ on $\partial\mathcal{C}_\theta$.

\begin{lemma}\label{Lemma-C-1-estimate}
Given $\phi\in\mathcal{O}$ and $\theta\in(0,\pi)$. Let $h$ be a positive, capillary even convex solution to \eqref{Robin-problem-of-the-Monge-Ampere-equation-normalized}, then there exists a constant $C_1>0$ depending only on $f,n,\theta,\phi,\|h\|_{C^0}$ such that
\begin{equation}\label{C-1-estimate}
|\nabla h|\leqslant C_1.
\end{equation}
\end{lemma}
\begin{proof}
We consider the following auxiliary function
\begin{equation*}
\Phi(\xi)=\log\left(\frac{|\nabla u(\xi)|^2}{2}\right)+u(\xi)+Kd(\xi),\ \xi\in\mathcal{C}_\theta,
\end{equation*}
where $K=1-2\cot\theta$. We choose a neighborhood $\mathcal{N}$ of the north pole, such that the closure $\overline{\mathcal{N}}\subset\mathcal{C}_\theta$. Assume $\Phi$ attains the maximum value at $\xi_0\in\mathcal{C}_\theta$.

{\bf Case 1.} If $\xi_0\in\partial\mathcal{C}_\theta$, then $\nabla_\mu\Phi(\xi_0)\geqslant0$. Using an orthonormal frame $\{e_i\}_{i=1}^n$ at $\xi_0$ with $e_n=\mu$, we have from \eqref{Preliminaries-formula-1}
\begin{equation*}
\begin{aligned}
\nabla_\mu\Phi&=\frac{2}{|\nabla u|^2}\sum_{i=1}^nu_iu_{in}+u_n+K\nabla_\mu d\\
&=-\frac{2\cot\theta}{|\nabla u|^2}\sum_{i=1}^nu_i^2-K\\
&=-1,
\end{aligned}
\end{equation*}
which shows a contradiction $-1\geqslant0$. Thus, this case is impossible.

{\bf Case 2.} If $\xi_0\in\mathcal{N}$, we can get the estimates of $|\nabla u(\xi_0)|$ via the standard interior gradient bound in Chou-Wang \cite{Chou-Wang-A_variational_theory_of_the_Hessian_equation}.

{\bf Case 3.} If $\xi_0\in\mathcal{C}_\theta\setminus(\partial\mathcal{C}_\theta\cup\mathcal{N})$, we choose an orthonormal frame $\{e_i\}_{i=1}^n$ at $\xi_0$ such that
\begin{equation*}
S_{ij}=\frac{h_{ij}+h\delta_{ij}}{\ell}=u_{ij}+\frac{1}{\ell}(u_i\ell_j+u_j\ell_i)+\frac{u}{\ell}\delta_{ij}
\end{equation*}
is diagonal. Then, we have
\begin{equation}\label{C-1-estimate-proof-formula-1}
S_{ii}=u_{ii}+\frac{2\cos\theta u_i\langle e_i,e\rangle+u}{\ell},
\end{equation}
and for $i\neq j$,
\begin{equation}\label{C-1-estimate-proof-formula-2}
0=S_{ij}=u_{ij}+\frac{\cos\theta u_i\langle e_j,e\rangle+\cos\theta u_j\langle e_i,e\rangle}{\ell}.
\end{equation}
Thus, we have at $\xi_0$
\begin{equation}\label{C-1-estimate-proof-formula-3}
0=\nabla_i\Phi=\frac{2}{|\nabla u|^2}\sum_{k=1}^nu_ku_{ki}+u_i+Kd_i.
\end{equation}

Define
\begin{equation*}
I=\left\{i\,:\,|u_i|\geqslant\frac{|\nabla u|}{\sqrt{n}}\right\}.
\end{equation*}
For $i\in I$, substituting \eqref{C-1-estimate-proof-formula-2} into \eqref{C-1-estimate-proof-formula-3}, there holds
\begin{equation*}
\begin{aligned}
u_{ii}&=-\frac{1}{2}|\nabla u|^2-\frac{Kd_i|\nabla u|^2}{2u_i}-\frac{1}{u_i}\sum_{k\neq i}u_ku_{ki}\\
&=-\frac{1}{2}|\nabla u|^2-\frac{Kd_i|\nabla u|^2}{2u_i}+\frac{\cos\theta}{\ell u_i}\sum_{k\neq i}u_k\big(u_k\langle e_i,e\rangle+u_i\langle e_k,e\rangle\big)\\
&\leqslant-\frac{1}{2}|\nabla u|^2+\frac{\sqrt{n}K|\nabla d|}{2}|\nabla u|+\frac{\cos\theta(1+\sqrt{n})}{\ell}|\nabla u|.
\end{aligned}
\end{equation*}
Since $S_{ij}$ is positive definite, from \eqref{C-1-estimate-proof-formula-1} we have
\begin{equation*}
\begin{aligned}
0&\leqslant S_{ii}=u_{ii}+\frac{2\cos\theta u_i\langle e_i,e\rangle+u}{\ell}\\
&\leqslant-\frac{1}{2}|\nabla u|^2+\frac{\sqrt{n}K|\nabla d|}{2}|\nabla u|+\frac{\cos\theta(1+\sqrt{n})}{\ell}|\nabla u|+\frac{1}{\ell}(2\cos\theta|\nabla u|+|u|)\\
&\leqslant-\frac{1}{2}|\nabla u|^2+\left(\frac{\sqrt{n}K}{2}\max_{\mathcal{C}_\theta\setminus\mathcal{N}}|\nabla d|+\frac{\cos\theta(3+\sqrt{n})}{1-\cos\theta}\right)|\nabla u|+\frac{1}{(1-\cos\theta)^2}C_0\\
&=:-\frac{1}{2}|\nabla u|^2+B|\nabla u|+C,
\end{aligned}
\end{equation*}
where we used the $C^0$-estimate \eqref{C-0-estimate}. This shows $|\nabla u|\leqslant\widetilde{C_1}$ for a positive constant $\widetilde{C_1}$ depending only on $\theta,n$ and $\|h\|_{C^0}$. Finally, we get
\begin{equation*}
|\nabla h|\leqslant|\nabla\ell|u+\ell|\nabla u|\leqslant\frac{|\cos\theta|}{1-\cos\theta}C_0+2\widetilde{C_1}=:C_1,
\end{equation*}
This completes the proof of this lemma.
\end{proof}

\subsection{$C^2$-estimate}

Let $\{e_i\}_{i=1}^n$ be an orthonormal frame on $\mathcal{C}_\theta$, we denote $h_{ij}=\nabla^2h(e_i,e_j)$, $W_{ij;k}=\nabla_{e_k}(W_{ij})$ and $W_{ij;kl}=\nabla_{e_l}\nabla_{e_k}(W_{ij})$ etc.

\begin{lemma}\label{Lemma-C-2-estimate-1}
Given $\phi\in\mathcal{O}$ and $\theta\in(0,\frac{\pi}{2})$. Let $h$ be a positive, capillary even convex solution to \eqref{Robin-problem-of-the-Monge-Ampere-equation-normalized}, then there exists a constant $C_2'>0$ depending only on $f,n,\theta,\phi,\|h\|_{C^0}$ such that
\begin{equation}\label{C-2-estimate-1}
\max_{\mathcal{C}_\theta}|\nabla^2h|\leqslant\max_{\partial\mathcal{C}_\theta}|\nabla^2_{\mu\mu}h|+C_2'.
\end{equation}
\end{lemma}
\begin{proof}
We consider the function
\begin{equation*}
P(\xi,\Xi)=\nabla^2_{\Xi,\Xi}h(\xi)+h(\xi),
\end{equation*}
for $\xi\in\mathcal{C}_\theta$ and the unit vector $\Xi\in T_\xi\mathcal{C}_\theta$. Suppose that $P$ attains its maximum at $\xi_0\in\mathcal{C}_\theta$ and $\Xi_0\in T_\xi\mathcal{C}_\theta$. We divide the proof into the following two cases:

{\bf Case 1.} $\xi_0\in\mathcal{C}_\theta\setminus\partial\mathcal{C}_\theta$. We choose an orthonormal frame $\{e_i\}_{i=1}^n$ around $\xi_0$, such that $(W_{ij})=(h_{ij}+h\delta_{ij})$ is diagonal and $\Xi_0=e_1$. Without loss of generality, we assume $h_{11}>0$. Denote
\begin{equation}\label{C-2-estimate-proof-formula-1}
F(W_{ij})=\log\det(W_{ij})=\log f+\log|\widehat{\Sigma}|-\log h-\log\phi\left(\frac{\ell}{h}\right),
\end{equation}
and
\begin{equation*}
F^{ij}=\frac{\partial F}{\partial W_{ij}},\ \ F^{ij,kl}=\frac{\partial^2F}{\partial W_{ij}\partial W_{kl}}.
\end{equation*}
By taking twice covariant derivatives in the direction $e_1$ to \eqref{C-2-estimate-proof-formula-1}, and using the concavity of $F$, we have
\begin{equation*}
\begin{aligned}
\sum_{i,j=1}^nF^{ij}W_{ij;11}&=-\sum_{i,j,k,l=1}^nF^{ij,kl}W_{ij;1}W_{kl;1}+\widetilde{f}_{11}-\frac{h_{11}h-h^2_1}{h^2}
+\frac{\ell\phi'\left(\frac{\ell}{h}\right)}{h^2\phi\left(\frac{\ell}{h}\right)}h_{11}\\
&\quad+\frac{\big[\left(\phi'\left(\frac{\ell}{h}\right)\right)^2-\phi\left(\frac{\ell}{h}\right)\phi''\left(\frac{\ell}{h}\right)\big]
(\ell_1h-\ell h_1)^2}{h^4\phi^2\left(\frac{\ell}{h}\right)}
+\frac{2\phi'\left(\frac{\ell}{h}\right)\left(\ell_1h_1-\ell\frac{h_1^2}{h}\right)}{h^2\phi\left(\frac{\ell}{h}\right)}\\
&\geqslant\widetilde{f}_{11}+\left(\frac{\ell\phi'\left(\frac{\ell}{h}\right)}{h^2\phi\left(\frac{\ell}{h}\right)}-\frac{1}{h}\right)
h_{11}+\frac{2\phi'\left(\frac{\ell}{h}\right)\left(\ell_1h_1-\ell\frac{h_1^2}{h}\right)}{h^2\phi\left(\frac{\ell}{h}\right)},
\end{aligned}
\end{equation*}
where $\widetilde{f}=\log f$, and we used the log-concavity of $\phi$. Again, using the assumption $A_3$ in \eqref{Assumptions-phi}
\begin{equation*}
\frac{d}{dx}\log\frac{\phi(x)}{x}\geqslant0,\ \ \mbox{for}\ x>0,
\end{equation*}
and letting $x=\frac{\ell}{h}$, we get
\begin{equation*}
\frac{\ell\phi'\left(\frac{\ell}{h}\right)}{h^2\phi\left(\frac{\ell}{h}\right)}-\frac{1}{h}\geqslant0.
\end{equation*}
Combining the $C^0$-estimate \eqref{C-0-estimate} and the $C^1$-estimate \eqref{C-1-estimate}, we have
\begin{equation}\label{C-2-estimate-proof-formula-2}
\sum_{i,j=1}^nF^{ij}W_{ij;11}\geqslant\widetilde{f}_{11}+\frac{2\phi'\left(\frac{\ell}{h}\right)
\left(\ell_1h_1-\ell\frac{h_1^2}{h}\right)}{h^2\phi\left(\frac{\ell}{h}\right)}\geqslant\alpha,
\end{equation}
where $\alpha$ is a constant depending on $n,\theta,\phi,\|h\|_{C^0},\|h\|_{C^1},\|f\|_{C^2}$.

Using the Ricci identity
\begin{equation*}
h_{klij}=h_{ijkl}+2h_{kl}\delta_{ij}-2h_{ij}\delta_{kl}+h_{li}\delta_{kj}-h_{kj}\delta_{il},
\end{equation*}
we have
\begin{equation}\label{C-2-estimate-proof-formula-3}
\begin{aligned}
\sum_{i,j=1}^nF^{ij}h_{11ij}&=\sum_{i,j=1}^nF^{ij}(h_{ij11}+2h_{11}\delta_{ij}-2h_{ij}+h_{1i}\delta_{1j}-h_{1j}\delta_{i1})\\
&=\sum_{i,j=1}^nF^{ij}h_{ij11}+2h_{11}\sum_{i=1}^nF^{ii}-2\sum_{i,j=1}^nF^{ij}h_{ij}
\end{aligned}
\end{equation}
From \eqref{C-2-estimate-proof-formula-2}, we have
\begin{equation}\label{C-2-estimate-proof-formula-4}
\sum_{i,j=1}^nF^{ij}h_{ij11}=\sum_{i,j=1}^nF^{ij}(W_{ij;11}-h_{11}\delta_{ij})\geqslant\alpha-h_{11}\sum_{i=1}^nF^{ii}.
\end{equation}
Note that
\begin{equation}\label{C-2-estimate-proof-formula-5}
\sum_{i,j=1}^nF^{ij}h_{ij}=\sum_{i,j=1}^nF^{ij}(W_{ij}-h\delta_{ij})=n-h\sum_{i=1}^nF^{ii}.
\end{equation}
From the arithmetic-geometric mean inequality, we have
\begin{equation}\label{C-2-estimate-proof-formula-6}
\sum_{i=1}^nF^{ii}\geqslant n\left(\prod_{i=1}^nF^{ii}\right)^\frac{1}{n}=n(\det W_{ij})^{-\frac{1}{n}}
=n\left(\frac{f|\widehat{\Sigma}|}{h\phi\left(\frac{\ell}{h}\right)}\right)^{-\frac{1}{n}}\geqslant\gamma,
\end{equation}
where $\gamma>0$ is a constant depending only on $n,\theta,\phi,\|h\|_{C^0},\|f\|_{C^0}$.

At the maximum point $\xi_0$ of $P$, using \eqref{C-2-estimate-proof-formula-3}, \eqref{C-2-estimate-proof-formula-4}, \eqref{C-2-estimate-proof-formula-5}, \eqref{C-2-estimate-proof-formula-6}, we obtain
\begin{equation*}
\begin{aligned}
0&\geqslant\sum_{i,j=1}^nF^{ij}P_{ij}=\sum_{i,j=1}^nF^{ij}h_{11ij}+\sum_{i,j=1}^nF^{ij}h_{ij}\\
&=\sum_{i,j=1}^nF^{ij}h_{ij11}+2h_{11}\sum_{i=1}^nF^{ii}-\sum_{i,j=1}^nF^{ij}h_{ij}\\
&\geqslant\alpha-h_{11}\sum_{i=1}^nF^{ii}+2h_{11}\sum_{i=1}^nF^{ii}+h\sum_{i=1}^nF^{ii}-n\\
&\geqslant\gamma h_{11}+\alpha+\frac{\gamma}{C_0}-n,
\end{aligned}
\end{equation*}
that is,
\begin{equation*}
0<h_{11}\leqslant\frac{n-\alpha}{\gamma}-\frac{1}{C_0}=:C_2'.
\end{equation*}
This completes the proof in this case.

{\bf Case 2.} $\xi_0\in\partial\mathcal{C}_\theta$. In this case, we have the following two subcases:

{\bf Subcase 1.} $\Xi_0$ is tangential. We choose an orthonormal frame $\{e_i\}_{i=1}^n$ around $\xi_0\in\partial\mathcal{C}_\theta$, such that $e_n=\mu$ and $e_1=\Xi_0$ at $\xi_0$. By \eqref{Preliminaries-formula-2}, we have
\begin{equation*}
\begin{aligned}
h_{n11}&=(\nabla_{e_1}(\nabla^2h))(e_1,e_n)=\nabla_{e_1}(\nabla^2h(e_1,e_n))-\nabla^2h(\nabla_{e_1}e_1,e_n)-\nabla^2h(e_1,\nabla_{e_1}e_n)\\
&=\nabla_{e_1}(h_{1n})-\nabla^2h\left(\sum_{k=1}^{n-1}\Gamma_{11}^ke_k-\cot\theta e_n,e_n\right)-\nabla^2h(e_1,\cot\theta e_1)\\
&=\cot\theta(h_{nn}-h_{11}).
\end{aligned}
\end{equation*}
By the commutator formula for $3$-order covariant derivatives, we get
\begin{equation*}
h_{11n}=h_{1n1}+\sum_{k=1}^nh_kR_{k11n}=h_{n11}-h_n.
\end{equation*}
At the maximum point $\xi_0$ of $P$, we have
\begin{equation*}
0\leqslant\nabla_nP(\xi_0,e_1)=h_{11n}+h_n=\cot\theta(h_{nn}-h_{11}).
\end{equation*}
Since $0<\theta<\frac{\pi}{2}$, it implies
\begin{equation*}
h_{11}\leqslant h_{nn}\leqslant\max_{\partial\mathcal{C}_\theta}|\nabla^2_{\mu\mu}h|.
\end{equation*}

{\bf Subcase 2.} $\Xi_0$ is non-tangential. Write
\begin{equation*}
\Xi_0=a\tau+b\mu,\ a=\langle\Xi_0,\tau\rangle,\ b=\langle\Xi_0,\mu\rangle,
\end{equation*}
where $\tau$ is a unit tangential vector. Thus, $a^2+b^2=1$. By \eqref{Preliminaries-formula-2}, there holds at $\xi_0\in\partial\mathcal{C}_\theta$
\begin{equation*}
\nabla^2_{\Xi_0,\Xi_0}h=\nabla^2h(a\tau+b\mu,a\tau+b\mu)=a^2h_{\tau\tau}+b^2h_{\mu\mu}.
\end{equation*}
Hence, we obtain
\begin{equation*}
\begin{aligned}
P(\xi_0,\Xi_0)&=\nabla^2_{\Xi_0,\Xi_0}h(\xi_0)+h(\xi_0)\\
&=a^2h_{\tau\tau}+b^2h_{\mu\mu}+(a^2+b^2)h(\xi_0)\\
&=a^2P(\xi_0,\tau)+b^2P(\xi_0,\mu)\\
&\leqslant a^2P(\xi_0,\Xi_0)+b^2P(\xi_0,\mu),
\end{aligned}
\end{equation*}
which reduces
\begin{equation*}
P(\xi_0,\Xi_0)\leqslant P(\xi_0,\mu).
\end{equation*}
Hence, we get
\begin{equation*}
\nabla^2_{\Xi_0,\Xi_0}h\leqslant h_{\mu\mu}\leqslant\max_{\partial\mathcal{C}_\theta}|\nabla^2_{\mu\mu}h|.
\end{equation*}
This completes the proof of the estimate \eqref{C-2-estimate-1}.
\end{proof}

Let the function
\begin{equation*}
\zeta(\xi)=e^{-d(\xi)}-1,\ \xi\in\mathcal{C}_\theta,
\end{equation*}
where $d$ is the distance function in \eqref{The-distance-function}. Then, $\zeta$ has the following properties
\begin{equation}\label{Zeta-function-property}
\left\{
\begin{aligned}
&\zeta|_{\partial\mathcal{C}_\theta}=0,\\
&\nabla\zeta|_{\partial\mathcal{C}_\theta}=\mu,\\
&\nabla^2_{ij}\zeta\geqslant\frac{1}{2}\min\{\cot\theta,1\}\delta_{ij},\ \mbox{in}\ \Omega_\varepsilon=\{\xi\in\mathcal{C}_\theta\,|\,d(\xi)\leqslant\varepsilon\},
\end{aligned}
\right.
\end{equation}
where $\varepsilon$ is a sufficiently small constant. Please refer to \cite{Mei-Wang-Weng-The_capillary_Minkowski_problem}.

\begin{lemma}\label{Lemma-C-2-estimate-2}
Given $\phi\in\mathcal{O}$ and $\theta\in(0,\frac{\pi}{2})$. Let $h$ be a positive, capillary even convex solution to \eqref{Robin-problem-of-the-Monge-Ampere-equation-normalized}, then there exists a constant $C_2''>0$ depending only on $f,n,\theta,\phi,\|h\|_{C^0}$ such that
\begin{equation}\label{C-2-estimate-2}
\max_{\partial\mathcal{C}_\theta}|\nabla^2_{\mu\mu}h|\leqslant C_2''.
\end{equation}
\end{lemma}
\begin{proof}
We consider the auxiliary function
\begin{equation*}
Q(\xi)=\langle\nabla h,\nabla\zeta\rangle-\left(A+\frac{1}{2}M\right)\zeta(\xi)-\cot\theta h(\xi), \ \xi\in\Omega_\varepsilon,
\end{equation*}
where $A$ is a positive constant to be determined later, $M=\max_{\partial\mathcal{C}_\theta}|\nabla^2_{\mu\mu}h|$, $\varepsilon$ is a small constant.

Assume that $Q$ attains its minimum value at $\xi_0\in(\Omega_\varepsilon\setminus\partial\Omega_\varepsilon)$, we choose an orthonormal frame $\{e_i\}_{i=1}^n$ around $\xi_0$ such that $(W_{ij})=(h_{ij}+h\delta_{ij})$ is diagonal at $\xi_0$. By taking the covariant derivatives in the direction $e_k$ to \eqref{C-2-estimate-proof-formula-1}, we have
\begin{equation}\label{C-2-estimate-proof-formula-7}
\sum_{i,j=1}^nF^{ij}W_{ij;k}=\widetilde{f}_k+\left(\frac{\ell\phi'\left(\frac{\ell}{h}\right)}{h^2\phi\left(\frac{\ell}{h}\right)}
-\frac{1}{h}\right)h_k-\frac{\ell_k\phi'\left(\frac{\ell}{h}\right)}{h\phi\left(\frac{\ell}{h}\right)}\leqslant\lambda_1,
\end{equation}
where $\widetilde{f}=\log f$, and $\lambda_1>0$ is a constant depending only on $n,\theta,\phi,\|f\|_{C^1},\|h\|_{C^0}$ and $\|h\|_{C^1}$. Then, from \eqref{C-2-estimate-proof-formula-7}, \eqref{C-0-estimate}, \eqref{C-1-estimate} and \eqref{Zeta-function-property}, we obtain
\begin{equation*}
\begin{aligned}
0&\leqslant\sum_{i,j=1}^nF^{ij}Q_{ij}
=\sum_{i,j,k=1}^nF^{ij}h_{kij}\zeta_k+\sum_{i,j,k=1}^nF^{ij}h_k\zeta_{kij}+2\sum_{i,j,k=1}^nF^{ij}h_{ki}\zeta_{kj}\\
&\quad-\left(A+\frac{1}{2}M\right)\sum_{i,j=1}^nF^{ij}\zeta_{ij}-\cot\theta\sum_{i,j=1}^nF^{ij}h_{ij}\\
&=\sum_{i,k=1}^nF^{ii}(W_{ii;k}-h_i\delta_{ki})\zeta_k+\sum_{i,k=1}^nF^{ii}h_k\zeta_{kii}+2\sum_{i=1}^nF^{ii}(W_{ii}-h)\zeta_{ii}\\
&\quad-\left(A+\frac{1}{2}M\right)\sum_{i=1}^nF^{ii}\zeta_{ii}-\cot\theta\sum_{i=1}^nF^{ii}(W_{ii}-h)\\
&\leqslant\lambda_2\left(1+\sum_{i=1}^nF^{ii}\right)-\frac{1}{2}\left(A+\frac{1}{2}M\right)\min\{\cot\theta,1\}\sum_{i=1}^nF^{ii},
\end{aligned}
\end{equation*}
where $\lambda_2>0$ is a constant depending only on $\zeta,\lambda_1$. Let
\begin{equation}\label{C-2-estimate-proof-formula-8}
A=\frac{2}{\min\{\cot\theta,1\}}\left(\lambda_2+\frac{\lambda_2}{\gamma}\right)+1+\frac{C_1\max_{\Omega_\epsilon}|\nabla\zeta|
+C_0\cot\theta}{1-e^{-\epsilon}},
\end{equation}
then by \eqref{C-2-estimate-proof-formula-6}, we get
\begin{equation*}
\lambda_2\left(1+\sum_{i=1}^nF^{ii}\right)-\frac{1}{2}\left(A+\frac{1}{2}M\right)\min\{\cot\theta,1\}\sum_{i=1}^nF^{ii}<0.
\end{equation*}
This is a contradiction. Therefore, $\xi_0\in\partial\Omega_\varepsilon$. Next, we proceed with two cases: (1) If $\xi_0\in\partial\mathcal{C}_\theta$, it is easy to see that $Q(\xi_0)=0$; (2) If $\xi_0\in(\partial\Omega_\varepsilon\setminus\partial\mathcal{C}_\theta)$, then $d(\xi_0)=\varepsilon$. From \eqref{C-2-estimate-proof-formula-8}, we have
\begin{equation*}
Q(\xi_0)\geqslant-(\max_{\Omega_\varepsilon}|\nabla\zeta|)|\nabla h|+\left(A+\frac{1}{2}M\right)(1-e^{-\varepsilon})-\cot\theta h(\xi_0)>0.
\end{equation*}
In conclusion, we deduce that $Q(\xi)\geqslant0$ in $\Omega_\varepsilon$.

Assume $h_{\mu\mu}(\eta_0)=\max_{\partial\mathcal{C}_\theta}h_{\mu\mu}$ for some $\eta_0\in\partial\mathcal{C}_\theta$, then $\eta_0$ is clearly a minimum point of $Q$. Hence, from \eqref{Preliminaries-formula-2} we have
\begin{equation*}
\begin{aligned}
0&\geqslant\nabla_\mu Q(\eta_0)\\
&=\sum_{k=1}^n(h_{k\mu}\zeta_k+h_k\zeta_{k\mu})-\left(A+\frac{1}{2}M\right)\nabla_\mu\zeta(\eta_0)-\cot\theta\nabla_\mu h(\eta_0)\\
&\geqslant h_{\mu\mu}(\eta_0)-\left(A+\frac{1}{2}M\right)-\lambda_3,
\end{aligned}
\end{equation*}
where $\lambda_3>0$ is a constant depending only on $n,\theta,\|h\|_{C^0},\|h\|_{C^1}$. It implies
\begin{equation}\label{C-2-estimate-proof-formula-9}
\max_{\partial\mathcal{C}_\theta}h_{\mu\mu}\leqslant A+\lambda_3+\frac{1}{2}M.
\end{equation}
Similarly, we consider an auxiliary function as
\begin{equation*}
\overline{Q}(\xi)=\langle\nabla h,\nabla\zeta\rangle+\left(\overline{A}+\frac{1}{2}M\right)\zeta(\xi)-\cot\theta h(\xi), \ \xi\in\Omega_\epsilon,
\end{equation*}
where $\overline{A}$ is a positive constant. Adapting the similar argument as above, we know that $\overline{Q}(\xi)\leqslant0$ in $\Omega_\varepsilon$, and further
\begin{equation}\label{C-2-estimate-proof-formula-10}
\min_{\partial\mathcal{C}_\theta}h_{\mu\mu}\geqslant-\overline{A}-\lambda_4-\frac{1}{2}M,
\end{equation}
for a constant $\lambda_4>0$ depending only on $n,\theta,\|h\|_{C^0},\|h\|_{C^1}$. Let $C_2''=2\max\{A+\lambda_3,\overline{A}+\lambda_4\}$. Combining \eqref{C-2-estimate-proof-formula-9} and \eqref{C-2-estimate-proof-formula-10}, we have
\begin{equation*}
\max_{\partial\mathcal{C}_\theta}|\nabla^2_{\mu\mu}h|\leqslant C_2''.
\end{equation*}
This completes the proof of this lemma.
\end{proof}

\subsection{High order estimates}

From the previous $C^0$-estimate, $C^1$-estimate and $C^2$-estimate, we can establish the high order estimate in this subsection.
\begin{lemma}
Given $\phi\in\mathcal{O}$ and $\theta\in(0,\frac{\pi}{2})$. Let $h$ be a positive, capillary even convex solution to \eqref{Robin-problem-of-the-Monge-Ampere-equation-normalized}. For any integer $m\geqslant1$ and $\alpha\in(0,1)$, there exists a constant $C>0$ depending only on $n,\theta,\phi,\|f\|_{C^{m+1}}$, such that
\begin{equation}\label{High-order-estimates}
\|\nabla h\|_{C^{m+1,\alpha}}\leqslant C.
\end{equation}
\end{lemma}
\begin{proof}
By Lemma \ref{Lemma-C-0-estimate}, Lemma \ref{Lemma-C-1-estimate}, Lemma \ref{Lemma-C-2-estimate-1} and Lemma \ref{Lemma-C-2-estimate-2}, we obtain the $C^2$-estimate
\begin{equation*}
\|\nabla h\|_{C^2}\leqslant C.
\end{equation*}
By the theory of fully nonlinear second-order uniformly elliptic equations with oblique boundary conditions (see, e.g., \cite{Lions-Trudinger-Urbas-The_Neumann_problem_for_equations_of_Monge_Ampere_type}), we have
\begin{equation*}
\|\nabla h\|_{C^{2,\alpha}}\leqslant C.
\end{equation*}
Finally, the standard bootstrap techniques \cite{Gilbarg-Trudinger-book} imply the desired estimates \eqref{High-order-estimates}.
\end{proof}

\section{The proof of Theorem \ref{Main-Theorem}}\label{The proof of Theorem}

In this section, we use the continuity method to complete the proof of Theorem \ref{Main-Theorem}. Let
\begin{equation*}
f_t=(1-t)\phi(1)\ell+tf,\ \ \mbox{for}\ \ 0\leqslant t\leqslant 1.
\end{equation*}
Consider the problem
\begin{equation}\label{Robin-problem-of-the-Monge-Ampere-equation-f-t}
\left\{
\begin{aligned}
\det(h_{ij}+h\delta_{ij})&=\frac{1}{h\phi\left(\frac{\ell}{h}\right)}f_t,\ \ \ \ \mbox{in}\ \mathcal{C}_\theta,\\
\nabla_\mu h&=\cot\theta h,\ \ \ \ \ \ \ \ \mbox{on}\ \partial\mathcal{C}_\theta.
\end{aligned}
\right.
\end{equation}
Define the set
\begin{equation*}
\begin{aligned}
\mathcal{H}=&\bigg\{h\in C^{4,\alpha}(\mathcal{C}_\theta)\,:\,\nabla_\mu h=\cot\theta h\ \mbox{on}\ \partial\mathcal{C}_\theta,\ \mbox{and}\ \int_{\mathcal{C}_\theta}hv=0\\
&\quad\mbox{whenever}\ \int_{\mathcal{C}_\theta}\left(\frac{\ell\phi'\left(\frac{\ell}{h}\right)}{h\phi\left(\frac{\ell}{h}\right)}-n-1\right)
v\det(h_{ij}+h\delta_{ij})=0.\bigg\}
\end{aligned}
\end{equation*}
and we denote
\begin{equation*}
\mathcal{I}=\{t\in[0,1]\,:\,\mbox{Eq.}\ \eqref{Robin-problem-of-the-Monge-Ampere-equation-f-t}\ \mbox{has a positive even solution in}\ \mathcal{H}\}.
\end{equation*}

We define the nonlinear operator
\begin{equation*}
\mathcal{G}(h)=\det(h_{ij}+h\delta_{ij})-\frac{1}{h\phi\left(\frac{\ell}{h}\right)}f_t,\ h\in\mathcal{H}.
\end{equation*}
Thus, the linearized operator of $\mathcal{G}$ is
\begin{equation*}
L_h(v)=\sum_{i,j=1}^nc(W)_{ij}(v_{ij}+v\delta_{ij})+\frac{v}{h^2\phi\left(\frac{\ell}{h}\right)}
\left(1-\frac{\ell\phi'\left(\frac{\ell}{h}\right)}{h\phi\left(\frac{\ell}{h}\right)}\right)f_t,
\end{equation*}
where $c(W)_{ij}$ is the cofactor matrix of $W_{ij}=h_{ij}+h\delta_{ij}$.

\begin{lemma}\label{Lemma-symmetric-operator}
For $h,v,w\in\mathcal{H}$, there holds
\begin{equation*}
\int_{\mathcal{C}_\theta}wL_h(v)=\int_{\mathcal{C}_\theta}vL_h(w).
\end{equation*}
\end{lemma}
\begin{proof}
By Lemma 4.1 in \cite{Mei-Wang-Weng-The_capillary_Minkowski_problem}, there holds
\begin{equation*}
\int_{\mathcal{C}_\theta}w\sum_{i,j=1}^nc(W)_{ij}(v_{ij}+v\delta_{ij})=\int_{\mathcal{C}_\theta}v\sum_{i,j=1}^nc(W)_{ij}(w_{ij}+w\delta_{ij}).
\end{equation*}
Thus, we have
\begin{equation*}
\begin{aligned}
\int_{\mathcal{C}_\theta}wL_h(v)&=\int_{\mathcal{C}_\theta}w\left(\sum_{i,j=1}^nc(W)_{ij}(v_{ij}+v\delta_{ij})
+\frac{v}{h^2\phi\left(\frac{\ell}{h}\right)}\left(1-\frac{\ell\phi'\left(\frac{\ell}{h}\right)}
{h\phi\left(\frac{\ell}{h}\right)}\right)f_t\right)\\
&=\int_{\mathcal{C}_\theta}v\left(\sum_{i,j=1}^nc(W)_{ij}(w_{ij}+w\delta_{ij})+\frac{w}{h^2\phi\left(\frac{\ell}{h}\right)}
\left(1-\frac{\ell\phi'\left(\frac{\ell}{h}\right)}{h\phi\left(\frac{\ell}{h}\right)}\right)f_t\right)\\
&=\int_{\mathcal{C}_\theta}vL_h(w).
\end{aligned}
\end{equation*}
This completes the proof of this lemma.
\end{proof}

\begin{lemma}\label{Lemma-ker-space}
Let $\phi\in\mathcal{O}$ and $\theta\in(0,\frac{\pi}{2})$. Suppose $v\in\mathcal{H}$ and $v\in{\rm Ker}(L_h)$, then
\begin{equation*}
\int_{\mathcal{C}_\theta}\left(\frac{\ell\phi'\left(\frac{\ell}{h}\right)}{h\phi\left(\frac{\ell}{h}\right)}-n-1\right)
v\det(h_{ij}+h\delta_{ij})=0.
\end{equation*}
\end{lemma}
\begin{proof}
From $v\in{\rm Ker}(L_h)$ and \eqref{Robin-problem-of-the-Monge-Ampere-equation-f-t}, we have
\begin{equation*}
\sum_{i,j=1}^nc(W)_{ij}(v_{ij}+v\delta_{ij})+\frac{v}{h}\left(1-\frac{\ell\phi'\left(\frac{\ell}{h}\right)}
{h\phi\left(\frac{\ell}{h}\right)}\right)\det(h_{ij}+h\delta_{ij})=0.
\end{equation*}
Multiplying the above equality with $h$ and integrating over $\mathcal{C}_\theta$ and applying integration by parts twice, we get
\begin{equation}\label{Open-proof-formula-1}
\begin{aligned}
&\int_{\mathcal{C}_\theta}\left(\frac{\ell\phi'\left(\frac{\ell}{h}\right)}{h\phi\left(\frac{\ell}{h}\right)}-1\right)v\det(h_{ij}+h\delta_{ij})\\
=&\sum_{i,j=1}^n\int_{\mathcal{C}_\theta}h\,c(W)_{ij}(v_{ij}+v\delta_{ij})\\
=&n\int_{\mathcal{C}_\theta}v\det(h_{ij}+h\delta_{ij})+\sum_{i,j=1}^n\int_{\partial\mathcal{C}_\theta}c(W)_{ij}
(hv_i\langle\mu,e_j\rangle-vh_i\langle\mu,e_j\rangle).
\end{aligned}
\end{equation}
From \eqref{Preliminaries-formula-2}, we have
\begin{equation*}
(\nabla^2h+h\delta_{\mathbb{S}^n})(e_i,\mu)=0\ \ \mbox{on}\ \partial\mathcal{C}_\theta,\ 1\leqslant i\leqslant n-1.
\end{equation*}
Together with Robin boundary conditions of $h$ and $v$, we get
\begin{equation}\label{Open-proof-formula-2}
\sum_{i,j=1}^n\int_{\partial\mathcal{C}_\theta}c(W)_{ij}(hv_i\langle\mu,e_j\rangle-vh_i\langle\mu,e_j\rangle)=
\int_{\partial\mathcal{C}_\theta}\cot\theta c(W)_{nn}(hv-vh)=0.
\end{equation}
Combining \eqref{Open-proof-formula-1} with \eqref{Open-proof-formula-2}, we obtain
\begin{equation*}
\int_{\mathcal{C}_\theta}\left(\frac{\ell\phi'\left(\frac{\ell}{h}\right)}{h\phi\left(\frac{\ell}{h}\right)}-n-1\right)
v\det(h_{ij}+h\delta_{ij})=0.
\end{equation*}
This completes the proof of this lemma.
\end{proof}

Now, we give the proof of Theorem \ref{Main-Theorem}.
\begin{proof}[Proof of Theorem \ref{Main-Theorem}]
It is easy to see that $h=\ell$ is a solution to \eqref{Robin-problem-of-the-Monge-Ampere-equation-f-t} for $t=0$. The closeness of $\mathcal{I}$ comes from the a priori estimates \eqref{High-order-estimates}. For any solution $h$ to \eqref{Robin-problem-of-the-Monge-Ampere-equation-f-t}, Lemma \ref{Lemma-symmetric-operator} and Lemma \ref{Lemma-ker-space} imply that
\begin{equation*}
\begin{aligned}
&\text{Range}(L_h)=(\text{Ker}(L^*_h))^\bot=(\text{Ker}(L_h))^\bot\\
=&\bigg\{v\in\mathcal{H}\,:\,\int_{\mathcal{C}_\theta}\left(\frac{\ell\phi'\left(\frac{\ell}{h}\right)}
{h\phi\left(\frac{\ell}{h}\right)}-n-1\right)v\det(h_{ij}+h\delta_{ij})=0\bigg\}^\bot=\mathcal{H}.
\end{aligned}
\end{equation*}
That is, $L_h$ is surjective. Then, the openness of $\mathcal{I}$ follows from the implicit function theorem, and we obtain the existence of solutions to \eqref{Robin-problem-of-the-Monge-Ampere-equation}.

Next, we prove the uniqueness part. Let the support function $h$ of $\widehat{\Sigma}\in\mathcal{K}_\theta^\circ$ satisfy \eqref{Robin-problem-of-the-Monge-Ampere-equation}, and $|\widehat{\Sigma}|=1$. By the Orlicz-Minkowski inequality \eqref{Orlicz-Minkowski-inequality}, we have
\begin{equation*}
V_\phi(\widehat{\Sigma},\widehat{\mathcal{C}_\theta})\geqslant V(\widehat{\Sigma})\,
\phi\left(\frac{V(\widehat{\mathcal{C}_\theta})^\frac{1}{n+1}}{V(\widehat{\Sigma})^\frac{1}{n+1}}\right)
=\phi\left(|\widehat{\mathcal{C}_\theta}|^\frac{1}{n+1}\right).
\end{equation*}
Meanwhile, we note that
\begin{equation*}
V_\phi(\widehat{\Sigma},\widehat{\mathcal{C}_\theta})=\frac{1}{n+1}\int_{\mathcal{C}_\theta}\phi\left(\frac{\ell}{h}\right)h
\det(h_{ij}+h\delta_{ij})\,d\xi=\frac{1}{n+1}\int_{\mathcal{C}_\theta}f\,d\xi,
\end{equation*}
then
\begin{equation*}
\frac{1}{n+1}\int_{\mathcal{C}_\theta}f\geqslant\phi(|\widehat{\mathcal{C}_\theta}|^\frac{1}{n+1}).
\end{equation*}
If equality holds in above inequality and $\phi$ is strictly convex, then $\widehat{\Sigma}$ and $\widehat{\mathcal{C}_\theta}$ are dilates from the Orlicz-Minkowski inequality \eqref{Orlicz-Minkowski-inequality}. That is, $\Sigma$ is the spherical cap $|\widehat{\mathcal{C}_\theta}|^{-\frac{1}{n+1}}\mathcal{C}_\theta$.
\end{proof}

\vskip 5mm \noindent  {\bf Acknowledgement.} This paper is supported in part by the NSFC (No. 12371060), Shaanxi Fundamental Science Research Project for Mathematics and Physics (No. 22JSZ012).

\vskip 2mm \noindent Xudong Wang, \ \ \ {\small \tt xdwang@snnu.edu.cn}\\
{ \em School of Mathematics and Statistics, Shaanxi Normal University, Xi'an, 710119, China}\\

\vskip 2mm \noindent Baocheng Zhu, \ \ \ {\small \tt bczhu@snnu.edu.cn}\\
{ \em School of Mathematics and Statistics, Shaanxi Normal University, Xi'an, 710119, China}\\


\vskip 3mm

\begin{thebibliography}{99}

\addtolength{\itemsep}{-1.5ex}

\bibitem{Aleksandrov-surface_area_measure} A.D. Aleksandrov, On the theory of mixed volumes. III. Extension of two theorems of Minkowski on convex polyhedra to arbitrary convex bodies, Mat. Sbornik N.S., 3 (1938) 27--46.

\bibitem{Bianchi-Boroczky-Colesanti-Yang-The_L_p_Minkowski_problem_for_-n_p_1} G. Bianchi, K.J. B\"{o}r\"{o}czky, A. Colesanti, D. Yang, The $L_p$-Minkowski problem for $-n<p<1$, Adv. Math., 341 (2019) 493--535.

\bibitem{Bianchi-Boroczky-Colesanti-Smoothness_in_the_L_p_Minkowski_problem_for_p_1} G. Bianchi, K.J. B\"{o}r\"{o}czky, A. Colesanti, Smoothness in the $L_p$ Minkowski problem for $p<1$, J. Geom. Anal., 30 (2020) 680--705.

\bibitem{Boroczky-Lutwak-Yang-Zhang-The_logarithmic_Minkowski_problem} K.J. B\"{o}r\"{o}czky, E. Lutwak, D. Yang, G. Zhang, The logarithmic Minkowski problem, J. Amer. Math. Soc., 26 (2013) 831--852.

\bibitem{Boroczky-Lutwak-Yang-Zhang-Zhao-The_Gauss_Image_Problem} K.J. B\"{o}ro\"{o}czky, E. Lutwak, D. Yang, G. Zhang, Y. Zhao, The Gauss image problem, Comm. Pure Appl. Math., 73 (2020) 1406--1452.

\bibitem{Bryan-Ivaki-Scheuer-A_unified_flow_approach_to_smooth_even_L_p_Minkowski_problems} P. Bryan, M.N. Ivaki, J. Scheuer, A unified flow approach to smooth, even $L_p$-Minkowski problems, Anal. PDE, 12 (2019) 259--280.

\bibitem{Caffarelli-Interior_W_Estimates} L.A. Caffarelli, Interior $W^{2,p}$ estimates for solutions of the Monge-Amp\`{e}re equation, Ann. of Math., 131 (1990) 135--150.

\bibitem{Caffarelli-viscosity_solutions} L.A. Caffarelli, A localization property of viscosity solutions to the Monge-Amp\`{e}re equation and their strict convexity, Ann. of Math., 131 (1990) 129--134.

\bibitem{Chen-Li-Zhu-On_the_L_p_Monge_Ampere_equation} S. Chen, Q.-R. Li, G. Zhu, On the $L_p$ Monge-Amp\`{e}re equation, J. Differential Equations, 263 (2017) 4997--5011.

\bibitem{Cheng-Yau-Regularity_Minkowski_Problem} S.-Y. Cheng, S.-T. Yau, On the regularity of the solution of the $n$-dimensional Minkowski problem, Comm. Pure Appl. Math., 29 (1976) 495--516.

\bibitem{Chou-Wang-A_variational_theory_of_the_Hessian_equation} K.-S. Chou, X.-J. Wang, A variational theory of the Hessian equation, Comm. Pure Appl. Math., 54 (2001) 1029--1064.

\bibitem{Chou-Wang-The_L_p_Minkowski_problem_and_the_Minkowski_problem_in_centroaffine_geometry} K.-S. Chou, X.-J. Wang, The $L_p$-Minkowski problem and the Minkowski problem in centroaffine geometry, Adv. Math., 205 (2006) 33--83.

\bibitem{Fenchel-Jessen} W. Fenchel, B. Jessen, Mengenfunktionen und konvexe k\"{o}rper, Danske Vid. Selskab. Mat.-Fys. Medd., 16 (1938) 1--31.

\bibitem{Firey-p_Means_of_convex_bodies} W.J. Firey, $p$-Means of convex bodies, Math. Scand. 10 (1962) 17--24.

\bibitem{Gardner-Hug-Weil-The_Orlicz_Brunn_Minkowski_theory} R.J. Gardner, D. Hug, W. Weil, The Orlicz-Brunn-Minkowski theory: a general framework, additions, and inequalities, J. Differential Geom., 97 (2014) 427--476.

\bibitem{Gardner-Hug-Weil-Ye-The_dual_Orlicz_Brunn_Minkowski_theory} R.J. Gardner, D. Hug, W. Weil, D. Ye, The dual Orlicz-Brunn-Minkowski theory, J. Math. Anal. Appl., 430 (2015) 810--829.

\bibitem{Gilbarg-Trudinger-book} D. Gilbarg, N.S. Trudinger, Elliptic partial differential equations of second order, Springer, New York, 1983.

\bibitem{Guan-Lin-On_the_equation_det_on_S_n} P. Guan, C. Lin, On the equation $\det(u_{ij}+u\delta_{ij})=u^pf$ on $\mathbb{S}^n$, Preprint (1999).

\bibitem{Guang-Li-Wang-The_L_p_Minkowski_problem_with_super_critical_exponents} Q. Guang, Q.-R. Li, X.-J. Wang, The $L_p$-Minkowski problem with super-critical exponents, arXiv:2203.05099v1.

\bibitem{Guang-Li-Wang-Existence_of_convex_hypersurfaces_with_prescribed_centroaffine_curvature} Q. Guang, Q.-R. Li, X.-J. Wang, Existence of convex hypersurfaces with prescribed centroaffine curvature, Trans. Amer. Math. Soc., 377 (2024) 841--862.

\bibitem{Haberl-Lutwak-Yang-Zhang-Orlicz_Minkowski_problem} C. Haberl, E. Lutwak, D. Yang, G. Zhang, The even Orlicz Minkowski problem, Adv. Math., 224 (2010), 2485--2510.

\bibitem{He-Li-Wang-Multiple_solutions_of_the_Lp_Minkowski_problem} Y. He, Q.-R. Li, X.-J. Wang, Multiple solutions of the $L_p$-Minkowski problem, Calc. Var. Partial Differential Equations, 55 (2016) 13 pp.

\bibitem{Hu-Ivaki-Capillary_curvature_images} Y. Hu, M. N. Ivaki, Capillary curvature images, arXiv:2505.12921v1.

\bibitem{Hu-Ivaki-Scheuer-Capillary_Christoffel_Minkowski_problem} Y. Hu, M. N. Ivaki, J. Scheuer, Capillary Christoffel-Minkowski problem, arXiv:2504.09320v1.

\bibitem{Hu-Wei-Yang-Zhou-A_complete_family_of_Alexandrov_Fenchel_inequalities} Y. Hu, Y. Wei, B. Yang, T. Zhou, A complete family of Alexandrov-Fenchel inequalities for convex capillary hypersurfaces in the half-space, Math. Ann., 390 (2024) 3039--3075.

\bibitem{Huang-Lu-On_the_regularity_of_the_Lp_Minkowski_problem} Y. Huang, Q. Lu, On the regularity of the $L_p$ Minkowski problem, Adv. in Appl. Math., 50 (2013) 268--280.

\bibitem{Huang-Lutwak-Yang-Zhang-dual_curvature_measures} Y. Huang, E. Lutwak, D. Yang, G. Zhang, Geometric measures in the dual Brunn-Minkowski theory and their assciated Minkowski problems, Acta Math., 216 (2016) 325--388.

\bibitem{Huang-Lutwak-Yang-Zhang-L_p_Aleksandrov_Problem} Y. Huang, E. Lutwak, D. Yang, G. Zhang, The $L_p$-Aleksandrov problem for $L_p$-integral curvature, J. Differential Geom., 110 (2018) 29pp.

\bibitem{Huang-Yang-Zhang-Minkowski_problems_for_geometric_measures} Y. Huang, D. Yang, G. Zhang, Minkowski problems for geometric measures, Bull. Amer. Math. Soc., 62 (2025) 359--425.

\bibitem{Hug-Lutwak-Yang-Zhang-On_the_L_p_Minkowski_problem_for_polytopes} D. Hug, E. Lutwak, D. Yang, G. Zhang, On the $L_p$ Minkowski problem for polytopes, Discrete Comput. Geom., 33 (2005) 699--715.

\bibitem{Jian-Lu-Existence_of_solutions_to_the_Orlicz_Minkowski_problem} H. Jian, J. Lu, Existence of solutions to the Orlicz-Minkowski problem, Adv. Math., 344 (2019) 262--288.

\bibitem{Jian-Lu-Wang-Nonuniqueness_of_solutions_to_the_Lp_Minkowski_problem} H. Jian, J. Lu, X.-J. Wang, Nonuniqueness of solutions to the $L_p$-Minkowski problem, Adv. Math., 281 (2015) 845--856.

\bibitem{Lewy-differentialgeometricinlarge} H. Lewy, On differential geometric in the large. I. Minkowski's problem, Trans. Amer. Math. Soc., 43 (1938) 258--270.

\bibitem{Lions-Trudinger-Urbas-The_Neumann_problem_for_equations_of_Monge_Ampere_type} P.-L. Lions, N.S. Trudinger, J.I.E. Urbas, The Neumann problem for equations of Monge-Amp\`{e}re type, Comm. Pure Appl. Math., 39 (1986) 539--563.

\bibitem{Liu-Lu-A_flow_method_for_the_dual_Orlicz_Minkowski_problem} Y. Liu, J. Lu, A flow method for the dual Orlicz-Minkowski problem, Trans. Amer. Math. Soc., 373 (2020) 5833--5853.

\bibitem{Lu-Wang-Rotationally_symmetric_solutions_to_the_Lp_Minkowski_problem} J. Lu, X.-J. Wang, Rotationally symmetric solutions to the $L_p$-Minkowski problem, J. Differential Equations, 254 (2013) 983--1005.

\bibitem{Lutwak-The_Brunn_Minkowski_Firey_Theory_I} E. Lutwak, The Brunn-Minkowski-Firey Theory I: Mixed volumes and the Minkowski problem, J. Differential Geom., 38 (1993) 131--150.

\bibitem{Lutwak-Oliker} E. Lutwak, V. Oliker, On the regularity of solutions to a generalization of the Minkowski problem, J. Differential Geom., 41 (1995) 227--246.

\bibitem{Lutwak-Yang-Zhang-On_the_L_p_Minkowski_problem} E. Lutwak, D. Yang, G. Zhang, On the $L_p$-Minkowski problem, Trans. Amer. Math. Soc., 356 (2004) 4359--4370.

\bibitem{Lutwak-Yang-Zhang-Orlicz_projection_bodies} E. Lutwak, D. Yang, G. Zhang, Orlicz projection bodies, Adv. Math., 223 (2010) 220--242.

\bibitem{Lutwak-Yang-Zhang-Orlicz_centroid_bodies} E. Lutwak, D. Yang, G. Zhang, Orlicz centroid bodies, J. Differential Geom., 84 (2010) 365--387.

\bibitem{Lutwak-Yang-Zhang-Lp_Dual_curvature_measures} E. Lutwak, D. Yang, G. Zhang, $L_{p}$ dual curvature measures, Adv. Math., 329 (2018) 85--132.

\bibitem{Mei-Wang-Weng-A_constrained_mean_curvature_flow_and_Alexandrov_Fenchel_inequalities} X. Mei, G. Wang, L. Weng, A constrained mean curvature flow and Alexandrov-Fenchel inequalities, Int. Math. Res. Not. IMRN 1, (2024) 152--174.

\bibitem{Mei-Wang-Weng-The_capillary_Minkowski_problem} X. Mei, G. Wang, L. Weng, The capillary Minkowski problem, Adv. Math., 469 (2025) 29 pp.

\bibitem{Mei-Wang-Weng-Convex_capillary_hypersurfaces_of_prescribed_curvature_problem} X. Mei, G. Wang, L. Weng, Convex capillary hypersurfaces of prescribed curvature problem, arXiv:2504.14392v1.

\bibitem{Mei-Wang-Weng-The_capillary_L_p_Minkowski_problem} X. Mei, G. Wang, L. Weng, The capillary $L_p$-Minkowski problem, arXiv:2505.07746v2.

\bibitem{Mei-Wang-Weng-The_capillary_Gauss_curvature_flow} X. Mei, G. Wang, L. Weng, The capillary Gauss curvature flow, arXiv:2506.09840v1.

\bibitem{Mei-Wang-Weng-Xia-Alexandrov_Fenchel_inequalities_II} X. Mei, G. Wang, L. Weng, C. Xia, Alexandrov-Fenchel inequalities for convex hypersurfaces in the half-space with capillary boundary II, Math. Z., 310 (2025) 17 pp.

\bibitem{Mei-Weng-A_fully_nonlinear_locally_constrained_curvature_flow_for_capillary_hypersurface} X. Mei, L. Weng, A fully nonlinear locally constrained curvature flow for capillary hypersurface, Ann. Global Anal. Geom., 67 (2025) 17 pp.

\bibitem{Minkowski-konvexen_Polyeder} H. Minkowski, Allgemeine Lehrs\"{a}tze \"{u}ber die konvexen Polyeder, Nachr. Ges. Wiss. G\"{o}tt., (1897) 198--219.

\bibitem{Nirenberg} L. Nirenberg, The Weyl and Minkowski problems in differential geometry in the large, Comm. Pure Appl. Math., 6 (1953) 337--394.

\bibitem{Pogorelovbook} A.V. Pogorelov, The Minkowski multidimensional problem, V.H. Winston, Washington, D.C., 1978.

\bibitem{Schneider-book} R. Schneider, Convex bodies: the Brunn-Minkowski theory. 2nd edn., Encyclopedia of Mathematics and Its Applications, vol. 151, Cambridge University Press, Cambridge, (2014).

\bibitem{Shenfeld-van-Handel-Mixed_volumes_and_the_Bochner_method} Y. Shenfeld, R. van Handel, Mixed volumes and the Bochner method, Proc. Amer. Math. Soc., 147 (2019) 5385--5402.

\bibitem{Shenfeld-van-Handel-The_extremals_of_Minkowski_quadratic_inequality} Y. Shenfeld, R. van Handel, The extremals of Minkowski's quadratic inequality, Duke Math. J., 171 (2022) 957--1027.

\bibitem{Shenfeld-van-Handel-The_extremals_of_the_Alexandrov_Fenchel_inequality_for_convex_polytopes} Y. Shenfeld, R. van Handel, The extremals of the Alexandrov-Fenchel inequality for convex polytopes, Acta Math., 231 (2023) 89--204.

\bibitem{Wang-Weng-Xia-Alexandrov_Fenchel_inequalities_I} G. Wang, L. Weng, C. Xia, Alexandrov-Fenchel inequalities for convex hypersurfaces in the half-space with capillary boundary, Math. Ann., 388 (2024) 2121--2154.

\bibitem{Wang-Weng-Xia-A_Minkowski_type_inequality_for_capillary_hypersurfaces_in_a_half_space} G. Wang, L. Weng, C. Xia, A Minkowski-type inequality for capillary hypersurfaces in a half-space, J. Funct. Anal., 287 (2024) 22 pp.

\bibitem{Xi-Jin-Leng-The_Orlicz_Brunn_Minkowski_inequality} D. Xi, H. Jin, G. Leng, The Orlicz Brunn-Minkowski inequality, Adv. Math., 260 (2014) 350--374.

\bibitem{Zhu-Zhou-Xu-Dual_Orlicz_Brunn_Minkowski_theory} B. Zhu, J. Zhou, W. Xu, Dual Orlicz-Brunn-Minkowski theory, Adv. Math., 264 (2014) 700--725.

\bibitem{Zhu-The_logarithmic_Minkowski_problem_for_polytopes} G. Zhu, The logarithmic Minkowski problem for polytopes, Adv. Math., 262 (2014) 909--931.

\bibitem{Zhu-The_Lp_Minkowski_problem_for_polytopes_for_0_p_1} G. Zhu, The $L_p$ Minkowski problem for polytopes for $0<p<1$, J. Funct. Anal., 269 (2015) 1070--1094.

\bibitem{Zhu-The_centro_affine_Minkowski_problem_for_polytopes} G. Zhu, The centro-affine Minkowski problem for polytopes, J. Differential Geom., 101 (2015) 159--174.

\bibitem{Zou-Xiong-Orlicz_John_ellipsoids} D. Zou, G. Xiong, Orlicz-John ellipsoids, Adv. Math., 265 (2014) 132--168.

\end{thebibliography}
\end{document}